\documentclass[11pt,final]{amsart}
\usepackage[utf8]{inputenc}
\usepackage[T1]{fontenc}
\usepackage[letterpaper,centering]{geometry}
\usepackage{lmodern}
\usepackage{mathrsfs}
\usepackage{amsmath,amssymb,amsfonts}
\usepackage{xcolor}
\definecolor{darkblue}{rgb}{0,0,0.3}
\definecolor{darkgreen}{rgb}{0,0.4,0}
\usepackage[colorlinks=true,linkcolor=darkblue, citecolor=darkgreen, unicode,pdfborder={0 0 0}]{hyperref}
\usepackage{enumerate}
\usepackage{enumitem}

\usepackage[ps,arrow,matrix,tips,line,curve]{xy}
{\setbox0\hbox{$ $}}\fontdimen16\textfont2=\fontdimen17\textfont2
\entrymodifiers={+!!<0pt,\the\fontdimen22\textfont2>}
\SelectTips{cm}{11}

\usepackage{nicematrix}

\theoremstyle{plain}
\newtheorem{thm}{Theorem}[section]
\newtheorem{conj}[thm]{Conjecture}

\newtheorem{lem}[thm]{Lemma}
\newtheorem{cor}[thm]{Corollary}
\newtheorem{prop}[thm]{Proposition}

\theoremstyle{definition}

\newtheorem{define}[thm]{Definition}
\newtheorem{rem}[thm]{Remark}

\newtheorem{example}[thm]{Example}

\newtheorem{notate}[thm]{Notation}

\numberwithin{equation}{section}

\input cyracc.def
\DeclareFontFamily{U}{russian}{}
\DeclareFontShape{U}{russian}{m}{n}
        { <5><6> wncyr5
        <7><8><9> wncyr7
        <10><10.95><12><14.4><17.28><20.74><24.88> wncyr10 }{}
\DeclareSymbolFont{Russian}{U}{russian}{m}{n}
\DeclareSymbolFontAlphabet{\mathcyr}{Russian}
\makeatletter
\let\@math@cyr\mathcyr
\renewcommand{\mathcyr}[1]{\@math@cyr{\cyracc #1}}
\makeatother
\newcommand{\Sha}{{\mathcyr{Sh}}}
\newcommand{\Ba}{{\mathcyr{B}}}

\newcommand{\ab}{{\mathrm{ab}}}
\newcommand{\isoto}{\myxrightarrow{\,\sim\,}}
\makeatletter
\def\myrightarrow{{\setbox\z@\hbox{$\rightarrow$}\dimen0\ht\z@\multiply\dimen0 6\divide\dimen0 10\ht\z@\dimen0\box\z@}}
\def\myrightarrowfill@{\arrowfill@\relbar\relbar\myrightarrow}
\newcommand{\myxrightarrow}[2][]{\ext@arrow 0359\myrightarrowfill@{#1}{#2}}
\makeatother

\newcommand{\A}{{\mathrm A}}
\newcommand{\B}{{\mathrm B}}
\def\C{{\mathrm C}}

\newcommand{\N}{{\mathrm N}}
\renewcommand{\P}{{\mathrm P}}

\newcommand{\mS}{{\mathrm S}}

\newcommand{\T}{{\mathrm T}}

\newcommand{\Z}{{\mathrm Z}}

\newcommand{\ZZ}{{\mathbf Z}}
\newcommand{\QQ}{{\mathbf Q}}
\newcommand{\FF}{{\mathbf F}}
\newcommand{\RR}{{\mathbf R}}

\newcommand{\nr}{\mathrm{nr}}

\newcommand{\Gal}{\mathrm{Gal}}
\newcommand{\Pic}{\mathrm{Pic}}
\newcommand{\Br}{\mathrm{Br}}

\renewcommand{\emptyset}{\varnothing}

\newcommand{\Ker}{{\mathrm{Ker}}}

\newcommand{\Hom}{{\mathrm{Hom}}}
\newcommand{\mmu}{\boldsymbol{\mu}}

\DeclareMathOperator{\inv}{inv}

\DeclareMathOperator{\spec}{Spec}

\DeclareMathOperator{\Ext}{Ext}

\DeclareMathOperator{\Aut}{Aut}

\DeclareMathOperator{\SL}{SL}
\DeclareMathOperator{\cyc}{cyc}

\DeclareMathOperator{\im}{Im}
\DeclareMathOperator{\Mat}{Mat}

\def\U{\mathrm{U}}

\def\<{\langle}
\def\>{\rangle}

\def\alp{\alpha}
\def\bet{\beta}
\def\sig{\sigma}

\def\Gam{\Gamma}
\def\gam{\gamma}
\def\vphi{\varphi}
\def\eps{\varepsilon}
\def\Lam{\Lambda}

\def\lrar{\to}
\def\hrar{\hookrightarrow}
\def\ovl{\overline}
\def\Om{\Omega}
\def\x{\stackrel}

\def\wtl{\widetilde}
\def\del{\delta}
\def\what{\widehat}

\newcommand{\bark}{{\mkern1mu\overline{\mkern-1mu{}k\mkern-1mu}\mkern1mu}}
\newcommand{\GL}{\mathrm{GL}}

\newcommand{\rvline}{\hspace*{-\arraycolsep}\vline\hspace*{-\arraycolsep}}

\setlist[itemize]{leftmargin=*}
\setlist[enumerate]{leftmargin=*}

\date{April 13th, 2019; revised on December, 9th, 2021}
\title{The Massey vanishing conjecture for number fields}

\author{Yonatan Harpaz}
\address{Institut Galil\'ee, Universit\'e Sorbonne Paris Nord, 99~avenue Jean-Baptiste Cl\'ement, 93430 Villetaneuse, France}
\email{harpaz@math.univ-paris13.fr}

\author{Olivier Wittenberg}
\address{Institut Galil\'ee, Universit\'e Sorbonne Paris Nord, 99~avenue Jean-Baptiste Cl\'ement, 93430 Villetaneuse, France}
\email{wittenberg@math.univ-paris13.fr}

\begin{document}
\begin{abstract}
A conjecture of Mináč and Tân predicts that
for any $n\geq 3$, any prime~$p$ and any field~$k$,
the Massey product of~$n$ Galois cohomology classes in $H^1(k,\ZZ/p\ZZ)$
must vanish if it is defined.
We establish this conjecture when~$k$ is a number field.
\end{abstract}

\maketitle

\section{Introduction}

A powerful invariant of a topological space $X$ is its graded cohomology ring $H^*(X,R)$ with coefficients in a given ring $R$. It is a natural question to understand what geometric information about $X$ is captured by $H^*(X,R)$, and what information is lost. With that in mind, one may first observe that the cup product on $H^*(X,R)$ exists already on the level of the \emph{cochain complex} $C^*(X,R)$, where it satisfies the Leibniz rule with respect to the differential. In other words, $C^*(X,R)$ is a \emph{differential graded ring}. As a preliminary step one may thus ask exactly what is lost in the passage from $C^*(X,R)$ to $H^*(X,R)$. 

To make this information measurable, one is led to construct invariants of $C^*(X,R)$ which do not only depend on the associated cohomology ring. A prominent example of such invariants are the \emph{higher Massey products} \cite{Mas58}. To define these, let us fix a differential graded ring, i.e., a cochain complex $(C^\bullet,\partial)$ equipped with an associative graded product $C^n \otimes C^m \lrar C^{n+m}$ which satisfies the signed Leibniz rule with respect to $\partial$. Let $\alp_0,\dots,\alp_{n-1} \in H^1(C)$ be cohomology elements of degree $1$. A \emph{defining system} for the $n$-fold Massey product of $\alp_0,\dots,\alp_{n-1}$ is a collection of elements $a_{i,j} \in C^1$ for $0 \leq i < j \leq n,(i,j) \neq (0,n)$, satisfying the following two conditions\footnote{We follow Kraines' \cite{Kra66} definition, which differs from Dwyer's \cite{Dwy75} by a sign.}:
\begin{enumerate}
\item
$\partial a_{i,j} = -\sum_{m=i+1}^{j-1} a_{i,m} \cup a_{m,j}$. In particular, $\partial a_{i,i+1} = 0$.
\item
$[a_{i,i+1}] = \alp_{i} \in H^1(C)$ for $i=0,\dots,n-1$.
\end{enumerate}
Given a defining system $\Lam = \{a_{i,j}\}$, the element $b_{0,n} = -\sum_{m=1}^{n-1}a_{0,m} \cup a_{m,n}$ is then a $2$\nobreakdash-cocycle; the value $\<\alp_0,\dots,\alp_{n-1}\>_{\Lam} := [b_{0,n}] \in H^2(C)$ is called the \emph{$n$-fold Massey product} of $\alp_0,\dots,\alp_{n-1}$ with respect to the defining system $\Lam$. We then let $\<\alp_0,\dots,\alp_{n-1}\> \subseteq H^2(C)$ denote the set of all elements which can be obtained as the $n$-fold Massey product of $\alp_0,\dots,\alp_{n-1}$ with respect to some defining system. We say that the $n$-fold Massey product of $\alp_0,\dots,\alp_{n-1}$ is \emph{defined} if there exists a defining system as above, and that it \emph{vanishes} if it is defined and furthermore $\<\alp_0,\dots,\alp_{n-1}\>_{\Lam} = 0$ for at least one defining system $\Lam$.

Higher Massey products provide invariants of a differential graded ring which do not factor through its cohomology. A geometric example, originally described by Massey \cite{Mas58bis} \cite{Mas98}, where higher Massey products give non-trivial information is the following: let $I$ be a finite set and consider an embedding $\vphi\colon \coprod_{i \in I} \mathbf S^1 \hrar \RR^3$ of $|I|$ circles in $3$-space. Let $X := \RR^3 \setminus \im(\vphi)$ be its complement. Alexander duality shows that $H^1(X,\ZZ)$ and $H^2(X,\ZZ)$ are both naturally isomorphic to $\ZZ^I$. In addition, the cup product on $H^*(X,\ZZ)$ encodes the \emph{linking number} of each pair of circles. When these linking numbers are non-zero the cohomology ring of the complement $X$ detects the fact that the circle configuration is not isotopic to a trivial configuration (whose complement is homotopy equivalent to a wedge of circles and $2$-spheres). A case where this information is not enough to distinguish a given circle configuration from the trivial one is the \emph{Borromean rings}: a configuration of three circles in $\RR^3$ such that every pair is unlinked and yet the configuration as a whole is not isotopic to a trivial embedding. By contrast, the complement $X$ of the Borromean ring configuration can be distinguished from the complement of the trivial embedding by considering the \emph{differential graded ring} $C^*(X,\ZZ)$, instead of merely the associated cohomology ring. Indeed, since each two of the rings are unlinked, the cohomology group $H^1(X,\ZZ)$ contains three elements $\alp_0,\alp_1,\alp_2$, such that the cup product of each pair vanishes, and so there exists a defining system $\Lam$ for the associated triple Massey product. Furthermore, one can show that $\<\alp_0,\alp_1,\alp_2\>_{\Lam} \in H^2(X,\ZZ)$ is non-zero for any choice of defining system $\Lambda$. In particular, $X$ is not homotopy equivalent to the complement of the trivial configuration. 

Shifting attention from topology to arithmetic, it has been observed that open subschemes of spectra of number rings obtained by removing a finite set of primes $S$ behave much like complements of links in $\RR^3$. In particular, higher Massey products in étale cohomology of such schemes carries important information about their fundamental groups, which, in turn, controls Galois extensions unramified outside $S$ (see \cite{morishita}, \cite{Vog04}, \cite{Sha11}). By contrast, it was shown by Hopkins and Wickelgren \cite{HW15} that all triple Massey products \emph{vanish} if we replace the étale cohomology of Dedekind rings by the Galois cohomology of \emph{number fields}, and consider coefficients in $\FF_2$.
This result was extended to all fields by Mináč and Tân~\cite{MN17a}, still with coefficients in $\FF_2$.
The idea was then put forward that for every $n \geq 3$, every prime $p$ and every field, $n$-fold
Massey products with coefficients in $\FF_p$ should vanish as soon as they
are defined (first in
\cite{MN17a} under an assumption on the roots of unity and later in \cite{MN16} in general).
This conjecture became known as the \emph{Massey vanishing conjecture}:

\begin{conj}[Mináč, Tân]\label{c:massey}
For every field $k$, prime $p$ and cohomology classes $\alp_0,\dots,\alp_{n-1} \in H^1(k,\FF_p)$ with $n \geq 3$, if the subset $\<\alp_0,\dots,\alp_{n-1}\> \subseteq H^2(k,\FF_p)$ given by the $n$-fold Massey product is non-empty, then it contains $0 \in H^2(k,\FF_p)$.
\end{conj}

Matzri~\cite{Mat14}, followed by Efrat and Matzri~\cite{efratmatzri} and by Mináč and Tân \cite{MN16},
established this conjecture for $n=3$, yielding, in effect, a strong restriction on the type of groups which can appear as absolute Galois groups. This restriction is related to many subtle structural properties of the maximal pro-$p$-quotients of absolute Galois groups, see e.g., the work of Efrat~\cite{efrat2014} and Mináč--Tân~\cite{MN17b}. Recent work of Matzri~\cite{Mat18, Mat19} also yields vanishing results for certain triple Massey products of cohomology classes of degree~$>1$.

\begin{rem}\label{r:formal}
The Massey vanishing property appearing in Conjecture~\ref{c:massey} (Massey products of $n \geq 3$ classes are trivial as soon as they are defined) holds, for example, for differential graded rings which are \emph{formal}, that is, which are quasi-isomorphic to their cohomology.
Conjecture~\ref{c:massey} may then lead one to ask whether the differential graded ring $C^*(k,\FF_p)$ controlling Galois cohomology is actually formal---this question was indeed posed by Hopkins and Wickelgren, see \cite[Question 1.3]{HW15}. However, the answer is in the negative: Positselski \cite{positselski2011}
(see also \cite[\textsection6]{Pos17}) showed that for local fields of characteristic $\neq p$ which contain a primitive $p$-th root of unity (or a square root of $-1$ when $p=2$), the differential graded ring $C^*(k,\FF_p)$ is not formal (even though these fields do have the Massey vanishing property, see \cite{MN17a}).
\end{rem}

The Massey vanishing conjecture is known in the following cases:
\begin{enumerate}
\item
When $n=3$ and $k$ and $p$ are arbitrary \cite{Mat14, efratmatzri, MN16}.
\item
When $k$ is a local field, $n \geq 3$ and $p$ is arbitrary. We note that in this case, if $k$ has characteristic $p$ or does not possess a primitive $p$-th root of unity then $H^2(k,\FF_p)=0$ by \cite[Theorem 9.1]{Koc13} and local duality respectively, and hence the Massey vanishing property holds trivially. The claim in the general non-trivial case was proven in \cite[Theorem 4.3]{MN17a}, using local duality.
\item
When $k$ is a number field, $n=4$ and $p=2$ \cite{GMTW16}. (The proof exploits the arithmetic of
a splitting variety constructed in \emph{op.\ cit.}\ for $n=4$ and $p=2$.  Under genericity assumptions
on the classes $\alp_1,...,\alp_n$, this construction of a splitting variety was later
generalised to $n=4$ and arbitrary~$p$ in~\cite{GM19}.)
\item
For all $n$ and all $p$, when $k$ is a field of virtual cohomological dimension 1 \cite{PS18}.
\end{enumerate}

\medskip
Our main result in the present paper is the following (see Theorem~\ref{t:massey} below):

\begin{thm}
The Massey vanishing conjecture holds for all number fields, all $n \geq 3$ and all primes $p$. 
\end{thm}

Our strategy can be summarized as follows. We begin in \S\ref{s:dwyerpalschlank} by recalling an equivalent formalism for $n$-fold Massey products in group cohomology due to Dwyer \cite{Dwy75}, which is based on non-abelian cohomology. We then combine this with the work of Pál and Schlank \cite{PS16} on the relation between homogeneous spaces and embedding problems in order to conclude that a certain naturally occurring homogeneous space $V$ with finite geometric stabilizers can serve as a \emph{splitting variety} for $n$-fold Massey products. These finite stabilizers are \emph{supersolvable}: they admit a finite filtration by normal subgroups, invariant under the natural outer Galois action, whose consecutive quotients are cyclic. This means that $V$ is susceptible to the machinery developed by the authors \cite{HW18} for proving the existence of rational points defined over number fields. As a result, the problem of the $n$-fold Massey vanishing property for number fields can be reduced to that of local fields plus an additional constraint: one needs to show that the local Massey solutions can be chosen to satisfy \emph{global reciprocity} in terms of the unramified Brauer group of $V$. This unramified Brauer group turns out to play a more prominent role specifically when $n$ is small. This is the main topic of \S\ref{s:brauer-1}, where some crucial properties of the unramified Brauer group are extracted in this case. The proof of the Massey vanishing conjecture is then given in \S\ref{s:main}, see Theorem~\ref{t:massey}.

We finish this paper with two additional sections. In \S\ref{s:brauer-2} we provide more elaborate computations of the Brauer group of $V$ when $n \leq 5$, and we describe examples showing that this Brauer group can be non-trivial. Then, in \S\ref{s:beyond} we discuss the Massey vanishing conjecture with more general types of coefficients. We show that for coefficients in the Tate twisted module $\ZZ/p\ZZ(i)$, the proof of the main theorem can be adapted with very small modifications to show the analogous Massey vanishing property over number fields (Theorem~\ref{t:massey-2}). We also show that for coefficients in $\ZZ/8$, the statement of the conjecture actually \emph{fails} for number fields. More precisely, we show that there exist three $\ZZ/8$\nobreakdash-characters of $\QQ$ whose triple Massey product is defined but does not vanish. It seems likely, though, that
the Massey vanishing property with coefficients in $\ZZ/m$ for any integer $m$ can still be shown to hold for \emph{local fields}, using~\cite[Theorem 9.1]{Koc13}
and local duality.

\bigskip
\emph{Acknowledgements.}
We are grateful to Pierre Guillot and J\'an Min\'a\v{c} for many discussions about Massey products,
to Adam Topaz for pointing out the splitting varieties for Massey products
that were constructed by Ambrus P\'al and Tomer Schlank in~\cite{PS16},
to Ido Efrat and Nguy\~{\^e}n Duy T\^an for their comments on a preliminary version of this article
and to the referees for their useful comments.

\section{Higher Massey products, non-abelian cohomology and splitting varieties}\label{s:dwyerpalschlank}

In this section, we explain how Conjecture~\ref{c:massey}, over number fields,
can be related with the problem of constructing an adelic point in the unramified Brauer--Manin set
of a certain homogeneous space.  To this end, we first describe a
group-theoretic approach to Massey products in group cohomology due to
Dwyer~\cite{Dwy75} (a tool already employed in the works of
Hopkins--Wickelgren and Mináč--Tân on Massey products, see \cite{HW15},
\cite{MN16}, \cite{MN17a}, \cite{MN17b}) and recall, following Pál and
Schlank \cite[\S 9]{PS16}, the construction of a homogeneous space of
$\SL_N$ which is a splitting variety for the resulting embedding problem.

Fix an integer $n \geq 2$ and let $\U$ be the group of upper triangular unipotent (=~unitriangular) $(n+1) \times (n+1)$-matrices with coefficients in $\FF_p$, with rows and columns indexed from $0$ to $n$. For $0 \leq i < j \leq n$ we denote by $e_{i,j} \in \U$ the unitriangular matrix whose $(i,j)$ entry is $1$ and all other non-diagonal entries are $0$. Let
\[ \{1\} = \U^{n} \subseteq \dots \subseteq \U^m \subseteq \dots \subseteq \U^0 = \U \]
be the lower central series of $\U$, so that $\U^{m+1} = [\U^m,\U]$ and $\Z := \U^{n-1} = \<e_{0,n}\>$ is the center of $\U$. We can identify $\U^m$ explicitly as the normal subgroup of $\U$ generated by the elementary matrices $e_{i,j}$ for $j-i > m$. Equivalently, it is the subgroup consisting of those unitriangular matrices whose first $m$ non-principal diagonals vanish. In particular, $\U^1 = \U'$ is the derived subgroup and $\U^3 = [\U',\U']$ is the second derived subgroup. 
We will denote by $\A = \U/\U^1$ the abelianization of $\U$. We have a natural basis $\{\ovl{e}_{i,i+1}\}_{i=0}^{n-1}$ for $\A$, where $\ovl{e}_{i,i+1}$ is the image of $e_{i,i+1}$. 

Let $k$ be a field with absolute Galois group $\Gam_k := \Gal(\ovl{k}/k)$. Using the basis $\{e_{i,i+1}\}$ above we may identify the data of a homomorphism $\alp\colon \Gam_k \lrar \A$ with that of a collection $\alp_{i} = [a_{i,i+1}] \in H^1(k,\FF_p)$ of cohomology classes for $i=0,\dots,n-1$.  Let $u \in H^2(\U/\Z,\Z)$ be the element classifying the central extension
\[ 1 \lrar \Z \lrar \U \lrar \U/\Z \lrar 1 \rlap{.}\]
We note that since $n \geq 2$ the abelianization map $\U \lrar \A$ sends the center to $0$ and so factors through a map $\U/\Z \lrar \A$. 

\begin{notate}\label{n:coset}
For $X \in \U/\Z$ and $0 \leq i \leq j \leq n$ with $(i,j) \neq (0,n)$, we will denote by $X_{i,j}$ the $(i,j)$ entry of any coset representative $M \in \U$ of $X$. (This does not depend on the choice of $M$ since multiplying a matrix in $\U$ by a matrix in $\Z$ does not change the $(i,j)$ entry of the former, for any $(i,j) \neq (0,n)$.)
\end{notate}

We now recall the following result of Dwyer.  (Note that a sign appears in \cite[Theorem~2.4]{Dwy75}, however we recall
that we are using a different convention in the definition of Massey products.)

\begin{prop}[{\cite{Dwy75}, see also~\cite[Proposition 8.3]{efrat2014}}]\label{p:dwyer}
Let $\alp_0,\dots,\alp_{n-1} \in H^1(k,\FF_p)$ be cohomology classes.
Write $\alp\colon \Gam_k \to \A$ for the corresponding group homomorphism. 
\begin{enumerate}
\item[(i)]
If $\beta\colon \Gam_k \to \U/\Z$ is a homomorphism lifting $\alp$,
then the collection $\{\beta_{i,j}\}$ of $1$-cochains
given by $\beta_{i,j}(\sigma)=\beta(\sigma)_{i,j}$
 (see Notation~\ref{n:coset}) forms a defining system for the $n$-fold Massey product of $\alp_0,\dots,\alp_{n-1}$. 
\item[(ii)]
The association $\beta \mapsto \{\beta_{i,j}\}$ puts group homomorphisms $\beta\colon\Gam_k \lrar \U/\Z$ lifting $\alp$ and defining systems for the $n$-fold Massey product of $\alp_0,\dots,\alp_{n-1}$ in bijection.
\item[(iii)]
The value of the $n$-fold Massey product of $\alp_0,\dots,\alp_{n-1}$ with respect to a defining system $\{\beta_{i,j}\}$ as in~(ii) is given by the class $-\bet^*u \in H^2(k,\Z) = H^2(k,\FF_p)$, where we have identified the center $\Z$ with $\FF_p$ via the generator $e_{0,n}$.
\end{enumerate}
\end{prop}

\begin{cor}\label{c:dwyer}\
\begin{enumerate}
\item
The homomorphism $\alp\colon \Gam_k \lrar \A$ lifts to $\U/\Z$ if and only if $\<\alp_0,\dots,\alp_{n-1}\> \neq \varnothing$.
\item
The homomorphism $\alp\colon \Gam_k \lrar \A$ lifts to $\U$ if and only if $\<\alp_0,\dots,\alp_{n-1}\>$ contains $0$. 
\item
The Massey vanishing conjecture is equivalent to the statement that any homomorphism $\alp\colon \Gam_k \lrar \A$ which lifts to $\U/\Z$ also lifts to $\U$.
\end{enumerate}
\end{cor}

\begin{rem}
In the setting of Corollary~\ref{c:dwyer}(3), the Massey vanishing conjecture does not require that \emph{every} lift of $\alp$ to $\U/\Z$ also lifts to $\U$.
\end{rem}

Let us fix a homomorphism $\alp\colon \Gam_k \lrar \A$.  Using the basis $\{e_{i,i+1}\}$ above, we will think of~$\alp$
as a tuple $\{\alp_{i}=[a_{i,i+1}]\}_{i=0}^{n-1}$ of elements of $H^1(k,\FF_p) = \Hom(\Gam_k,\FF_p)$. 
Consider the embedding problem depicted by the diagram of profinite groups
\begin{equation}
\begin{aligned}
\label{e:embedding}
\xymatrix@R=3ex{
& & & \Gam_k \ar^{\alp}[d] \ar@{-->}[dl] & \\
1 \ar[r] & \U^1 \ar[r] & \U \ar[r] & \A \ar[r] & 1 .\\
}
\end{aligned}
\end{equation}
Following Pál and Schlank \cite[\S 9]{PS16}, we associate with this embedding problem a homogeneous space $V$ over $k$, as follows.
Let $N=|\U|+1$ and consider $\SL_N$ as an algebraic group over $k$. Embed $\U$ into $\SL_{N}(k)$ via its augmented regular representation (where the additional dimension is used to fix the determinant). Let $T_\alp \lrar \spec(k)$ be the $k$-torsor under $\A$ determined by the homomorphism $\alp\colon\Gam_k \lrar \A$, viewed as a $1$-cocycle.
We let $V$ be the quotient variety
\begin{align}
\label{eq:defv}
V = (\SL_{N} \times T_\alp)/\U\rlap{,}
\end{align}
where~$\U$ acts on~$\SL_{N}$ by right multiplication,
on $T_\alp$ via the homomorphism $\U \lrar \A$,
and on $\SL_{N} \times T_\alp$, on the right, by the diagonal action.
The left action of $\SL_N$ on the first factor of the product $\SL_N \times T_\alp$ descends uniquely to $V$, exhibiting it as a homogeneous space of $\SL_N$ with geometric stabilizer $\U^1$.

By Corollary~\ref{c:dwyer}, we have that $0 \in \<\alp_0,\dots,\alp_{n-1}\>$ if and only if the dotted lift in~\eqref{e:embedding} exists, i.e., if and only if the embedding problem~\eqref{e:embedding} is solvable. On the other hand, by~\cite[Theorem 9.6]{PS16} the set of solutions to~\eqref{e:embedding} up to conjugacy by~$\U^1$ is in one-to-one correspondence with the set of $\SL_N(k)$-orbits of $V(k)$. We therefore conclude the following:

\begin{prop}[{\cite[Theorem 9.21]{PS16}}]\label{p:splitting}
The Massey product of $\alp_0,\dots,\alp_{n-1}$ contains~$0$ if and only if $V(k) \neq 0$. In other words, $V$ constitutes a \emph{splitting variety} for the $n$-fold Massey product of $\alp_0,\dots,\alp_{n-1}$.
\end{prop}

Using homogeneous spaces as splitting varieties opens the door to an application of the previous work of the authors \cite{HW18} 
to establish a local-global principle for the vanishing of Massey products, subject to a suitable Brauer--Manin obstruction. For this, recall first that if~$k$ has characteristic~$0$,
the \emph{unramified} Brauer group $\Br_{\nr}(V) \subseteq \Br(V)$ of a smooth $k$-variety $V$ is defined to be the image of $\Br(V^c)$ in $\Br(V)$, where $V^c$ is any choice of a smooth compactification of $V^c$ (the resulting subgroup of $\Br(V)$ is in fact independent of the choice of $V^c$). When~$k$ is a number field, there is a natural pairing
\[ \Br_{\nr}(V) \times \prod_{v \in \Om_k} V(k_v)  \to \QQ/\ZZ\text{,} \quad\quad (\beta,(x_v)) \mapsto \sum_{v \in \Om_k}\inv_v\beta(x_v)\rlap{,} \]
where $\Om_k$ is the set of places of $k$, and for every $v \in \Om_k$, we denote by $\inv_v\colon\Br(k_v) \to \QQ/\ZZ$ the canonical invariant map from class field theory.  
Here we note that $\beta$ belonging to the unramified Brauer group ensures that the sum on the right hand side has finitely many non-zero terms.
This pairing, also known as the Brauer--Manin pairing, was first defined by Manin \cite{Man71} for the purpose of constructing an obstruction to the local-global principle for rational points. More precisely, the global reciprocity law of class field theory implies that the value of this pairing vanishes on $(x_v)$ if the latter is the diagonal image of a $k$\nobreakdash-rational point $x$, and so the image of $V(k)$ in $\prod_{v \in \Om_k}V(k_v)$ is contained in the subspace of
 $\prod_{v \in \Om_k} V(k_v)$ consisting of those families that are orthogonal to $\Br_{\nr}(V)$ with
respect to the above pairing.  This subspace is referred to as the Brauer--Manin set of $V$. The non-emptiness of the Brauer--Manin set is then a necessary condition for the existence of $k$-rational points on $V$, a phenomenon now known as the \emph{Brauer--Manin obstruction}. 

\begin{thm}[{\cite[Théorème B]{HW18}}]
\label{t:supersolvable}
Let $k$ be a number field. Let $V$ a homogeneous space of a semi-simple and simply connected linear group $G$ with finite supersolvable geometric stabilizers, and let $V^c$ be a smooth compactification of $V$. Then $V^c(k)$ is dense in the Brauer--Manin set of $V^c$ with respect to the product of the $v$-adic topologies. 
In particular, if $V$ admits a collection of local points $(x_v) \in \prod_{v \in \Om_k}V(k_v)$ which is orthogonal $\Br_{\nr}(V)$, then $V$ has a $k$-rational point.
\end{thm}

In the statement of Theorem~\ref{t:supersolvable}, the supersolvability condition on the geometric stabilizers of $V$ takes into account the \emph{outer action} of $\Gam_k$ on these groups, 
induced from the (honest) conjugation action of the middle term on the left term in the short exact sequence
\begin{equation}\label{eq:fundamental-sequence} 
1 \to H_{\ovl{x}} \to G_{\ovl{x}} \to \Gam_k \to 1,
\end{equation}
where $H_{\ovl{x}}$ is the stabilizer of the geometric point $\ovl{x}$ and $G_{\ovl{x}} \subseteq G(\ovl{k}) \rtimes \Gam_k$ is the subgroup consisting of those pairs $(g,\sig)$ such that $g\sig(\ovl{x}) = \ovl{x}$, see~\cite[\textsection 2.3]{demarchelucchinireduction} for a discussion of three equivalent approaches to this construction. Namely, if $H_{\ovl{x}}$ is such a stabilizer,
then we require it to admit
a finite filtration
\[ \{1\} = H_0 \subseteq H_1 \subseteq ... \subseteq H_m = H_{\ovl{x}}\rlap{,} \]
such that each $H_i$ is a normal subgroup of $H_{\ovl{x}}$ stable under the outer Galois action, and each $H_{i+1}/H_i$ is cyclic, see \cite[Définition 6.4]{HW18}. 

\begin{rem}\label{r:change-base-point}
The stabiliser $H_{\ovl{x}}$ and the outer Galois action on it are independent of the choice of $\ovl{x}$ in the following sense.
Given another geometric point $\ovl{x}'$ and an $r \in G(\ovl{k})$ such that $r\ovl{x} = \ovl{x}'$, we obtain an isomorphism $G_{\ovl{x}} \to G_{\ovl{x}'}$  by mapping $(g,\sig)$
to $(rg\sig(r)^{-1},\sig)$. This isomorphism preserves the projection to $\Gam_k$, and hence the induced isomorphism $T\colon H_{\ovl{x}}\isoto H_{\ovl{x}'}$ preserves the outer Galois action. Though the choice of $r$ is not unique, any other $r$ sending $\ovl{x}$ to $\ovl{x}'$ will differ from it by an element of $H_{\ovl{x}'}$, and hence the resulting isomorphism $T$ will differ by an inner automorphism of $H_{\ovl{x}'}$. In particular $H_{\ovl{x}}$ and $H_{\ovl{x}'}$ are isomorphic as groups with outer Galois action, with an isomorphism that is canonical up to an inner isomorphism. 
\end{rem}

Let us now unwind the definitions to see what the sequence~\eqref{eq:fundamental-sequence} becomes in the case of the homogeneous $V = (\SL_N \times T_{\alp})/\U$ introduced in~\eqref{eq:defv}. Write 
\( \pi\colon \SL_N \times T_{\alp} \to V \)  
for the quotient map and let $\ovl{x} = \pi(1,z)$ be the image of a point of the form $(1,z) \in \SL_{N}(\ovl{k}) \times T_{\alp}(\ovl{k})$, so that the stabiliser of $\ovl{x}$ in $\SL_N(\ovl{k})$ coincides, by construction, with the embedded subgroup $\U^1 \subseteq \SL_{N}(k)$.
The short exact sequence~\eqref{eq:fundamental-sequence} can then be identified
with the pull-back along $\alp\colon \Gam_k \to A$ of the sequence appearing in~\eqref{e:embedding}, as is revealed by the computation
\begin{align*}
G_{\ovl{x}} & = \{(g,\sig) \in \SL_N(\ovl{k}) \times \Gam_k \;|\; g\sig(\ovl{x})
 = \ovl{x}\} \\
&= \{(g,\sig) \in \SL_N(\ovl{k}) \times \Gam_k \;|\; \pi(g,\sig(z)) = \pi(1,z)\}
\\ &=
\{(g,\sig) \in \U \times \Gam_k \;|\; [g] = \alp(\sig) \in A\} = \U \times_A \Gam_k\rlap,
\end{align*}
where $[g]$ stands for the class of $g \in \U$ in $A = \U/\U^1$. 
We deduce that the outer action on $\U^1$ as a geometric stabiliser in $V$ coincides
with the one
obtained by restricting along $\alp\colon \Gam_k \lrar \A$
the outer action of $\A=\U/\U^1$ on~$\U^1$
induced by the short exact sequence appearing in~\eqref{e:embedding}. 

\begin{prop}\label{p:fibration}
Assume that~$k$ is a number field.
The homogeneous space $V$ introduced in~\eqref{eq:defv} has a rational point if and only if it carries a collection of local points $(x_v) \in \prod_v V(k_v)$ which is orthogonal to the unramified Brauer group $\Br_{\nr}(V)$.
\end{prop}

\begin{proof}
By Theorem~\ref{t:supersolvable} and the preceding discussion,
it suffices to show that $\U^1$ is supersolvable as a group equipped with an outer Galois action, 
where the outer Galois action is the one 
obtained by restricting along $\alp\colon \Gam_k \lrar \A$
the outer action of $\A=\U/\U^1$ on~$\U^1$
induced by the short exact sequence in~\eqref{e:embedding}.
It therefore suffices to check that~$\U^1$ admits a filtration
\[1=H^l \subseteq \dots \subseteq H^1=\U^1\] 
by  
subgroups $H^i \subseteq \U^1$ such that
each~$H^i$ is normal in $\U$  
and each successive quotient $H^i/H^{i+1}$ is cyclic.
Now $\U$ is a nilpotent group with lower central series $\{\U^m\}_{m=0}^{n}$. 
In particular, if we consider $\U^{n} \subseteq \dots \subseteq \U^1$ as a filtration of $\U^1$, then each step is normal in~$\U$, each successive quotient is abelian, and the action of $\U$ on each successive quotient is trivial.  
This filtration can consequently be refined
to a filtration of the same nature with each successive quotient furthermore being cyclic.
\end{proof}

In order to exploit Proposition~\ref{p:fibration}
it is necessary to have some control over the unramified Brauer group $\Br_{\nr}(V)$.
A relatively more accessible part of $\Br_{\nr}(V)$
is the subgroup
\[ \Br_{1,\nr}(V) := \Ker[\Br_{\nr}(V) \lrar \Br(\ovl{V})] \subseteq \Br_{\nr}(V) \rlap{,} \]
consisting of those Brauer elements which vanish over the base change $\ovl{V} := V \otimes_k \ovl{k}$ of $V$ to the algebraic closure of $k$. We will refer to $\Br_{1,\nr}(V)$ as the algebraic unramified Brauer group. To control the Brauer--Manin obstruction coming from this part of the Brauer group, we establish in the next section an explicit formula for $\Br_{1,\nr}(V)$ when $V$ is a homogeneous space of a semi-simple simply connected algebraic group
with finite geometric stabilizers, following work of Harari~\cite{Har07}, Demarche~\cite{Dem10} and Lucchini Arteche~\cite{LA14}.

\section{Algebraic unramified Brauer groups of homogeneous spaces of semi-simple simply connected groups with finite stabilisers}

We fix, in this section, a field~$k$ of characteristic~$0$
with algebraic closure~$\ovl{k}$ and absolute Galois group $\Gam_k= \Gal(\ovl{k}/k)$, a homogeneous space~$V$ of a semi-simple
simply connected linear algebraic group over~$k$ (for instance of~$\SL_N$)
and a geometric point $\bar v \in V(\bark)$, and we let $\ovl{V} := V \otimes_k \ovl{k}$ be the base change of $V$ to $\ovl{k}$.  We assume that the stabilizer $H_{\bar v}$
of~$\bar v$ is finite.
We recall that there is a natural outer action of~$\Gam_k$ on~$H_{\bar v}$
(as described in~\S\ref{s:dwyerpalschlank})
and that the Cartier dual $\widehat{H}_{\bar v}^\ab=\Hom(H_{\bar v}^\ab,{\ovl{k}^*})=\Hom(H_{\bar v}^\ab,\mmu_{\infty})$
of its abelianization $H_{\bar v}^\ab$ is canonically isomorphic
to $\Pic(\ovl{V})$
as a $\Gam_k$\nobreakdash-module
(see \cite[\textsection5]{HW18};
here $\mmu_{\infty}$ denotes the torsion subgroup of~$\ovl{k}^*$).
As a consequence, the Hochschild--Serre spectral sequence provides an exact sequence
\begin{align}
\label{se:hs}
\Br(k) \lrar \Br_1(V) \lrar H^1(k,\widehat{H}_{\bar v}^\ab) \xrightarrow{\del} H^3(k,\ovl{k}^*)\rlap{.}
\end{align}
We denote by $\Br_0(V) \subseteq \Br_{1,\nr}(V) \subset \Br_1(V)$ the image of~$\Br(k)$.

\begin{rem}
\label{rem:deltavanishes}
The differential~$\delta$ vanishes when $V(k)\neq\emptyset$, or more generally when~$V$ possesses a $0$\nobreakdash-cycle
of degree~$1$ (see \cite[Remark~5.4.3]{ctskobook}), as well as when
$H^3(k,\ovl{k}^*)=0$, which occurs for instance when~$k$ is a number field or is the function
field of a curve over a number field (see \cite[Lemma~2.6]{ctpoonen}).
\end{rem}

The goal of this section is to describe,
in Proposition~\ref{p:formula} below, a formula for
the quotient group $\Br_{1,\nr}(V)/\Br_0(V)$, viewed as a subgroup of
$H^1(k,\widehat{H}_{\bar v}^\ab)$ via~\eqref{se:hs}. This yields an explicit description of the group $\Br_{1,\nr}(V)/\Br_0(V)$ at least when the differential $\delta$ vanishes, as will be the case when we apply the proposition in subsequent parts of the paper.
The formula we put forward builds upon a series of formulae established by
previous authors: by Harari~\cite[Proposition~4]{Har07} when $V(k)\neq\emptyset$ and~$k$
is a number field, by Demarche~\cite[Théorème~1]{Dem10} when~$V(k)$ is allowed to be empty but~$k$ is a number field,
and by Lucchini Arteche~\cite[Th\'eor\`eme~4.15]{LA14}
when $V(k)\neq\emptyset$ and~$k$ is arbitrary.

\begin{define}
For a group $H$ of exponent $d$, consider the action of $(\ZZ/d)^*$ on the set $H/(\mathrm{conjugacy})$ of conjugacy classes of $H$ via $[x] \mapsto [x^i]$ for $i \in (\ZZ/d)^*$. By the \emph{outer exponent} of $H$ we will mean the smallest divisor $e | d$ such that this action factors through the quotient map $(\ZZ/d)^* \lrar (\ZZ/e)^*$.
\end{define}

Let $e$ denote the outer exponent of $H_{\bar v}$.
We fix a finite
Galois subextension~$L/k$
of~$\ovl{k}/k$, with Galois group $G=\Gal(L/k)$, satisfying the following three conditions:
\begin{enumerate}
\item $L$ contains all $e$-th roots of unity;
\item $V(L)\neq\emptyset$;
\item the natural outer action of~$\Gamma_k$ on~$H_{\bar v}$ factors through $G=\Gal(L/k)$.
\end{enumerate}
In particular, the action of $\Gam_k$ on both $H_{\bar v}^\ab$ and $\widehat{H}_{\bar v}^\ab$ factors through $G$: we may, and will, consider
these two abelian groups
as $G$\nobreakdash-modules
and view $H^1(G,\widehat{H}_{\bar v}^\ab)$ as a subgroup of $H^1(k,\widehat{H}_{\bar v}^\ab)$ via
the inflation map
$H^1(G,\widehat{H}_{\bar v}^\ab) \hookrightarrow H^1(k,\widehat{H}_{\bar v}^\ab)$.

We let $\chi\colon G \lrar (\ZZ/e)^*$ be the homomorphism induced by the cyclotomic character
and remark that the
outer action of~$G$ on $H_{\bar v}$ induces an action of~$G$ on
the set
$H_{\bar v}/(\mathrm{conjugacy})$
 of conjugacy classes
of the group~$H_{\bar v}$.

Let $\sig \in G$ and $u \in H_{\bar v}$ satisfy the equality
$\sig[u] = [u^{\chi(\sig)}]$ in $H_{\bar v}/(\mathrm{conjugacy})$. Then the image $\ovl{u}$ of $u$ in $H_{\bar v}^{\ab}$ satisfies $\sig(\ovl{u}) = \chi(\sig)\ovl{u}$, and is hence $\sig$-invariant when considered as an element of the twisted Cartier dual $\Hom(\widehat{H}_{\bar v}^{\ab},\ovl{\mmu}_\infty)$ of $\widehat H_{\bar v}^{\ab} = \Hom(H_{\bar v}^{\ab},\mmu_{\infty})$, where $\ovl{\mmu}_\infty$
denotes the group of roots of unity in~$\ovl{k}^*$ equipped with the \emph{trivial}
Galois action.
(This bidual is canonically isomorphic to $H_{\bar v}^{\ab}$ as an abelian group,
with the Galois action twisted by $\chi$.)
In particular, if we let~$\ZZ$ act on
$\Hom(\widehat{H}_{\bar v}^{\ab},\ovl{\mmu}_\infty)$ via~$\sig$,
then
we may view $\ovl{u}$ as an element of $H^0(\ZZ,\Hom(\widehat{H}_{\bar v}^{\ab},\ovl{\mmu}_\infty))$.
We hence obtain a composite map 
\begin{align*}
\xymatrix@C=3.5em{
H^1(G,\widehat{H}_{\bar v}^\ab) \ar[r]^{\sigma^*} &
H^1(\ZZ,\widehat{H}_{\bar v}^\ab) \ar[r]^(.43){x\mapsto \ovl{u}\cup x} &
H^1(\ZZ,\ovl{\mmu}_\infty) = \ovl{\mmu}_\infty \rlap{,}
}
\end{align*}
where~$\sigma^*$ is the pull-back map along $\ZZ\to G$, $1\mapsto \sigma$.

\begin{prop}[Harari, Demarche, Lucchini Arteche]\label{p:formula}
There is an equality
\begin{align*}
\Br_{1,\nr}(V)/\Br_0(V) = \Big\{& \beta \in H^1(G,\widehat{H}_{\bar v}^\ab) \quad\Big|\quad  \del(\bet) = 0 \quad\text{and}\quad
\forall \sig\in G, \forall u \in H_{\bar v},\\
& \text{ if } \sig[u] = [u^{\chi(\sig)}] \text{ in } H_{\bar v}/(\mathrm{conjugacy})\text{, then } \ovl{u}\cup \sigma^*\bet = 1 \in \ovl{\mmu}_\infty\Big\}
\end{align*}
of subgroups of $H^1(k,\widehat{H}_{\bar v}^\ab)$, where~$\delta$ is the differential appearing in~\eqref{se:hs}.
\end{prop}

\begin{rem}
As explained in Remark~\ref{rem:deltavanishes},
the condition $\delta(\beta)=0$ appearing in the formula of
Proposition~\ref{p:formula} is vacuous in many cases
of interest.  In general, however, it cannot be dispensed with.
Although we do not provide the details here,
it is possible
to exhibit a homogeneous space~$V$ of~$\SL_N$ with finite (and abelian) geometric stabilizers,
over the field $k=\mathbf C((x))((y))((z))$,
and a class $\beta \in H^1(G, \widehat{H}_{\bar v}^\ab)$,
such that $\delta(\beta)\neq 0$ while~$\beta$ satisfies the other conditions
appearing in the formula of Proposition~\ref{p:formula}.
\end{rem}

In the present paper, we shall mostly apply Proposition~\ref{p:formula} in the following special case. In what follows, for a set $I$ equipped with an action of $G$ and an element $\sig \in G$, we write $I^{\sig}$ for the subset of $I$ fixed by $\sigma$. 
 
\begin{cor}\label{c:special-case}
Suppose that the cyclotomic character $\chi\colon G \to (\ZZ/e)^*$ is trivial. Then
the image of $\Br_{1,\nr}(V)$ in $H^1(k,\widehat{H}^{\ab}_v)$ coincides with
the set
of those $\bet \in H^1(G,\widehat{H}^{\ab}_v) \subseteq H^1(k,\widehat{H}^{\ab}_v)$
with $\del(\bet) = 0$ such that 
for every $\sig \in G$,
if we let~$\ZZ$ act on~$H_{\bar v}^{\ab}$ and on $\widehat{H}_{\bar v}^{\ab}$
through~$\sigma$,
the pull-back
 $\sigma^*\bet \in H^1(\ZZ,\widehat{H}^{\ab}_v)$
is orthogonal, with respect to the cup product pairing
\[ H^0(\ZZ,H_{\bar v}^{\ab}) \times H^1(\ZZ,\widehat{H}_{\bar v}^{\ab}) \to H^1(\ZZ,\ovl{\mmu}_\infty) = \ovl{\mmu}_\infty\rlap, \]
to the image of
the natural map
$(H_{\bar v}/(\mathrm{conjugacy}))^\sigma \to (H_{\bar v}^{\ab})^{\sigma} = H^0(\ZZ,H_{\bar v}^{\ab})$.
\end{cor}

\begin{proof}[Proof of Proposition~\ref{p:formula}]
Let $k'=k(V)$.
Let us consider the commutative diagram with exact rows
\begin{align*}
\xymatrix@R=3ex{
&& H^1(G,\widehat{H}_{\bar v}^\ab) \ar@{^{(}->}[d] &\\
\Br(k) \ar[r] \ar[d] & \Br_1(V) \ar[r] \ar[d] & H^1(k,\widehat{H}_{\bar v}^\ab) \ar@{^{(}->}[d]^(.4)\xi\ar[r]^{\del} & H^3(k,\ovl{k}^*) \ar[d] \\
\Br(k') \ar[r] & \Br_1(V_{k'}) \ar[r] & H^1(k',\widehat{H}_{\bar v}^\ab) \ar[r]^{\del'} & H^3(k',\ovl{k'}^*)\rlap{.}
}
\end{align*}
As~$k$ is algebraically closed in~$k'$, the map $\Gam_{k'} \to \Gam_k$ is surjective and the map $\xi$ in the above diagram is injective. 
We may then consider the group~$G$ as a quotient of~$\Gam_{k'}$
and the inflation map
$H^1(G,\widehat{H}_{\bar v}^\ab) \hookrightarrow H^1(k',\widehat{H}_{\bar v}^\ab)$
coincides with the composition of~$\xi$ with
the inflation map
$H^1(G,\widehat{H}_{\bar v}^\ab) \hookrightarrow H^1(k,\widehat{H}_{\bar v}^\ab)$.
Moreover,
as $V(k')\neq\emptyset$, the map~$\delta'$ vanishes.
Finally, an element of~$\Br(V)$ whose pull-back to~$\Br(V_{k'})$ belongs to the subgroup $\Br_{\nr}(V_{k'})$
itself belongs to the subgroup $\Br_{\nr}(V)$; indeed, 
a class of $\Br(V)$ is unramified
if and only if
its residue
along every irreducible divisor~$D$ of a smooth compactification of~$V$ vanishes,
and for any such~$D$,
 the natural map
$H^1(k(D),\QQ/\ZZ) \lrar H^1(k'(D), \QQ/\ZZ)$ is injective
since~$k(D)$ is algebraically closed in~$k'(D)$.
From all of these remarks, we conclude that in order to prove
Proposition~\ref{p:formula},
we may replace~$k$ and~$V$ with~$k'$ and~$V_{k'}$, and thus assume
that $V(k)\neq\emptyset$.

As the statement of Proposition~\ref{p:formula} does not depend on the choice of~$\bar v$ (see Remark~\ref{r:change-base-point}),
we may then assume that $\bar v \in V(k)$,
so that~$\Gam_k$ acts on~$H_{\bar v}$.
We are now in a position to apply~\cite[Th\'eor\`eme~4.15]{LA14}.
It only remains to check that the formulae given in Proposition~\ref{p:formula}
and in \emph{loc.\ cit.}\  agree.

To this end, we first note
that for $\sig \in \Gam_L$,
the map $N_\sig\colon H_{\bar v} \lrar H_{\bar v}^\ab$ appearing in \emph{loc.\ cit.}\ is
simply the quotient map, since the outer action of~$\Gamma_L$ on~$H_{\bar v}$ is trivial
and~$L$ contains all $e$\nobreakdash-th roots of unity.
Thus
\[\Br_{1,\nr}(V)/\Br_0(V) \subseteq 
\Ker\left(
H^1(k,\widehat{H}_{\bar v}^\ab)\lrar
H^1(L,\widehat{H}_{\bar v}^\ab)\right)=
H^1(G,\widehat{H}_{\bar v}^\ab)\]
as a consequence of \cite[Th\'eor\`eme~4.15]{LA14}.
Let us fix a class $\beta \in H^1(G,\widehat{H}_{\bar v}^\ab)$,
a cocycle $b\colon G \lrar \widehat{H}_{\bar v}^\ab$, $\sigma \mapsto b_\sigma$
representing~$\beta$,
and an element $u \in H_{\bar v}$.
We need to prove that the following two conditions are equivalent:
\begin{enumerate}
\item $\ovl{u} \cup \sig^*\bet = 1$
for all
$\sigma \in G$
such that $\sig[u] = [u^{\chi(\sig)}]$ in  $H_{\bar v}/(\mathrm{conjugacy})$;
\item $b_\sigma(N_\sigma(u))=1$
for all $\sigma \in G$.
\end{enumerate}
When $\sig[u] = [u^{\chi(\sig)}]$
in $H_{\bar v}/(\mathrm{conjugacy})$,
one has $N_\sigma(u)=\ovl{u}$ and $b_\sigma(\ovl{u})=\ovl{u} \cup \sig^*\bet$,
hence~(1) is equivalent to
\begin{enumerate}
\item[(2')] $b_\sigma(\ovl{u})=1$
for all $\sigma \in G$
such that $\sig[u] = [u^{\chi(\sig)}]$ in  $H_{\bar v}/(\mathrm{conjugacy})$
\end{enumerate}
and~(2) implies~(2').
It remains to prove~(2), assuming~(2').  
Let~$n$ be the smallest positive integer such that
 $\sig^n[u] = [u^{\chi(\sig^n)}]$
in $H_{\bar v}/(\mathrm{conjugacy})$.
By~(2'), one has $b_{\sigma^n}(\ovl{u})=1$.
On the other hand, one has
$b_\sigma(N_\sigma(u))=((1+\sigma+\dots+\sigma^{n-1})b_\sigma)(\ovl{u})=b_{\sigma^n}(\ovl{u})$,
where the first equality stems from
the definition of~$N_\sigma$
and the second one from the cocycle condition.
Hence~(2) must hold too.
\end{proof}

\section{The Brauer group of the Massey splitting varieties}\label{s:brauer-1}

Let~$k$ be a field of characteristic~$0$.  We take up the notation of
\textsection\ref{s:dwyerpalschlank}; in particular,~$V$ is the homogeneous
space~\eqref{eq:defv} associated with a fixed homomorphism $\alpha\colon\Gam_k \lrar \A$. Our goal in this section is to study the unramified Brauer group $\Br_{\nr}(V)$. For the most part we will be using the formula of Proposition~\ref{p:formula} in order to understand the algebraic part of $\Br_{1,\nr}(V)$. However, before we do so, let us first take a closer look at
the \emph{transcendental} part of $\Br_{\nr}(V)$, namely, the image of $\Br_{\nr}(V)$ in the unramified Brauer group $\Br_{\nr}(\ovl{V})$ of the base change $\ovl{V} := V \otimes_k \ovl{k}$ to the algebraic closure. It is known that the transcendental Brauer group can play a non-trivial role in the local-global principle when $k$ is a number field, 
see, e.g.,~\cite{DLAN15}. Fortunately, in our case we can
deduce from the works of Bogomolov and Michailov that it simply vanishes.

\begin{prop}\label{p:trans}
The group $\Br_{\nr}(\ovl{V})$ is trivial.
\end{prop}
\begin{proof}
By the classical work of Bogomolov~\cite{Bog88} we know that $\Br_{\nr}(\ovl{V})$ admits a purely group theoretical description in terms of the stabilizer $\U^1$. More precisely, there is an isomorphism
\begin{equation}\label{e:bog}
\Br_{\nr}(\ovl{V}) \cong \Ker\left[H^2(\U^1,\QQ/\ZZ) \lrar \prod_{H\subseteq \U^1} H^2(H,\QQ/\ZZ)\right] 
\end{equation}
where the product ranges over all \emph{bicyclic subgroups} $H \subseteq \U^1$ (equivalently, over all abelian subgroups).
The right hand side of~\eqref{e:bog} is also known as the \emph{Bogomolov multiplier} of~$\U^1$. By~\cite[Theorem 3.1]{Mic13}, the Bogomolov multiplier of $\U^1$ vanishes, hence the result. 
\end{proof}

Next let us turn our attention to the algebraic unramified Brauer group $\Br_{1,\nr}(V)$.
Let $\B = \U^1/[\U^1,\U^1] = \U^1/\U^3$ denote the abelianization of $\U^1$. We have a natural basis $\{\ovl{e}_{i,j}\}$ for $\B$ indexed by the set of ordered pairs $(i,j) \in \{0,\dots,n\}^2$ such that $2 \leq j-i \leq 3$, where $\ovl{e}_{i,j}$ is the image of the element $e_{i,j} \in \U^1$. 
We note that the action of $\A$ on $\B$ induced by the outer action of $\A$ on $\U^1$ via conjugation can be described very explicitly in terms of the basis above: for $\ovl{e}_{i,j} \in \A$ and $\ovl{e}_{k,l} \in \B$ we have 
\begin{equation}\label{e:T} 
\ovl{e}_{i,j}(\ovl{e}_{k,l}) = \left\{\begin{matrix}
\ovl{e}_{k,l} + \ovl{e}_{i,l} & j=k \\ 
\ovl{e}_{k,l} - \ovl{e}_{k,j} & i=l \\
\ovl{e}_{k,l} & j\neq k, i \neq l \rlap{.}
\end{matrix}\right.  
\end{equation}
We consider $\B$ as a Galois module via the map $\alp\colon \Gam_k \lrar \A$, and let $\what{\B} = \Hom(\B,\mmu_{\infty})$ denote the Cartier dual module.
The action of $\U$ on $\U^1$ via conjugation induces an action of $\U$ on $\U^1/(\mathrm{conjugacy})$. By construction this action becomes trivial when restricted to~$\U^1$ and hence descends to a well-defined action of $\A = \U/\U^1$ on $\U^1/(\mathrm{conjugacy})$. We note that this action is compatible with the action of $\A$ on $\B = (\U^1)^{\ab}$ described in~\eqref{e:T}, in the sense that the natural map of sets $\U^1/(\mathrm{conjugacy}) \lrar \B$ is $\A$-equivariant.

Let $e$ be the outer exponent of $\U^1$ and let $k \subseteq L \subseteq \ovl{k}$ be the minimal Galois subextension such that $\alp|_{\Gam_L} = 0$ and $L$ contains all $e$-th roots of unity. In other words, $L$ is the fixed field of the subgroup 
\[\Gam_L := \Ker\Big[\Gam_k \xrightarrow{(\alp,\chi)} A \times (\ZZ/e)^*\Big]\rlap,\]
where $\chi$ is the cyclotomic character encoding the Galois action on the $e$-th roots of unity.  
We denote by $G = \Gal(L/k)$ the (finite) Galois group of $L$ over $k$, so that
$\alp$ and~$\chi$ together determine an embedding 
\begin{align}
\label{eq:embeddingalpchi}
 (\alp,\chi): G \hookrightarrow \A \times (\ZZ/e)^* \rlap,
\end{align}
whose components we still call $\alp$ and $\chi$.
We also note that the actions of $\Gam_k$ on both~$\B$ and~$\what{\B}$ factor through $G$, and so we may consider $\B$ and $\what{\B}$ as $G$-modules. 

As $\alp|_{\Gam_L} = 0$, we have
$T_\alp(L)\neq\emptyset$, where $T_\alp$ is as in~\eqref{eq:defv};
hence $V(L)\neq\emptyset$. Proposition~\ref{p:formula} is therefore applicable in the present setting with this choice of~$L$ and~$G$,
noting that $H_{\bar v}^\ab = \B$. Using it, we identify $\Br_{1,\nr}(V)/\Br_0(V)$ with a subgroup of $H^1(G,\what{\B})$.

\begin{prop}
\label{p:neukirch}
If $k$ does not contain a primitive $p$-th root of unity then $H^1(G,\what{\B}) = 0$. In particular, in this case one has $\Br_{1,\nr}(V)/\Br_0(V) = 0$.
\end{prop}

\begin{proof}
If $k$ does not contain a primitive $p$-th root of unity then $p$ must be odd, and we note that $e$ is a positive power of $p$, since it divides the exponent of $\U^1$, 
which is a power of~$p$, but is also divisible by the exponent of $\B= (\U^1)^{\ab}$, which is $p$. The condition that $k$ does not contain a primitive $p$-th root of unity implies that 
the composed map 
\[\chi_p\colon G \x{\chi}{\lrar} (\ZZ/e)^* \lrar (\ZZ/p)^*\] 
is non-zero. 
Its kernel is a $p$\nobreakdash-group since it is contained in $A \times \Ker((\ZZ/e)^*\twoheadrightarrow (\ZZ/p)^*)$
(see~\eqref{eq:embeddingalpchi}),
which is a $p$-group,
while its image is of order prime to $p$, being contained in $(\ZZ/p)^*$. 
On the other hand, by~\eqref{eq:embeddingalpchi},
the group $G$ is abelian.
Hence~$G$ splits as the product $G = \Ker(\chi_p) \times \im(\chi_p)$. 
We deduce that $G$ contains $0 \neq \im(\chi_p) \subseteq (\ZZ/p)^*$ as a normal  
subgroup which acts on the $\FF_p$-vector space $\what{\B}$ by $a(v) = av$. Then $H^i(\im(\chi_p),\what{\B}) = 0$ for all $i \geq 0$ and so the desired result follows from the Hochschild--Serre spectral sequence.
\end{proof}

\begin{define}
Let $G$ be a group and $M$ a $G$-module. We set
\[ \Sha^1_{\cyc}(G,M) = \Ker\Bigg[H^1(G,M) \lrar \prod_{C\subseteq G} H^1(C,M)\Bigg]\rlap{,}\]
where the product is taken over all cyclic subgroups $C\subseteq G$.
\end{define}

\begin{prop}\label{p:sha-cyc}
If the differential~$\delta$ appearing in~\eqref{se:hs} vanishes (e.g., if~$k$ is a number field),
then there is a sequence of inclusions
\begin{align}
\Sha^1_{\cyc}(G,\what{\B}) \subseteq \Br_{1,\nr}(V)/\Br_0(V) \subseteq H^1(G,\what{\B})
\end{align}
of subgroups of $H^1(k,\what{\B})$.
\end{prop}
\begin{proof}
This follows immediately from the formula in Proposition~\ref{p:formula}.
\end{proof}

\begin{rem}\label{r:sha-cyc-is-ba}
It follows from Chebotarev's theorem that when~$k$ is a number field, the image of $\Sha^1_{\cyc}(G,\what{\B})$ in $H^1(k,\what{B})$ consists of those elements whose image in $H^1(k_v,\what{B})$ vanishes for almost all places~$v$. Under the isomorphism $H^1(k,\what{B}) \cong \Br_1(V)/\Br_0(V)$ coming from the Hochschild-Serre spectral sequence in this case, this image gets identified with the subgroup
$\Ba_{\omega}(V) \subseteq \Br_{1}(V)/\Br_0(V)$ consisting of the classes of those Brauer elements whose image in $\Br(V_{k_v})$ belongs to $\Br_0(V_{k_v})$ for almost all places $v$. Such Brauer elements are also known as \emph{locally constant almost everywhere}, and are always unramified. In particular, we may consider $\Ba_{\omega}(V)$ as a subgroup of $\Br_{1,\nr}(V)/\Br_0(V)$.
\end{rem}

Going back to an arbitrary field~$k$ of characteristic~$0$, our goal,
until the end of this section,
is to understand the subgroup $\Br_{1,\nr}(V)/\Br_0(V)$
of $H^1(G,\what{\B})$
in the case where $n \leq 6$. For this, it will be convenient to introduce some additional notation.
Given a matrix~$M$, we denote its coefficient in position $(i,j)$ by $M_{i,j}$.
For $r \geq 0$, we denote by $\N^r$ the set of those $(n+1)\times (n+1)$\nobreakdash-matrices
with coefficients in~$\FF_p$ all of whose $(i,j)$ entries with $i>j-r$ are zero. Let~$I$
denote the identity matrix.
It is straightforward that
\begin{itemize}
\item
$\N^r \cdot \N^s \subseteq \N^{r+s}$ for $r,s \geq 0$;
\item
$\N^r = 0$ for $r > n$;
\item
for $M \in \U$ and $r \geq 1$, one has $M \in \U^r$ if and only if $M-I \in \N^{r+1}$. 
\end{itemize}

\begin{lem}\label{l:exponent-0}
The exponent (and hence the outer exponent) of $\U^1$ divides $p^m$
for any $m$ such that $2p^m > n$.
\end{lem}
\begin{proof}
Consider $M \in \U^1$, so that $M-I \in \N^2$. 
Inside the $\FF_p$-algebra of all $(n+1) \times (n+1)$ matrices, the powers of $M$ span a commutative $\FF_p$-subalgebra, within which we compute 
\begin{align}
\label{eq:newton}
M^{p^m}= (I + (M-I))^{p^m} = I + (M-I)^{p^m} \rlap{.}
\end{align}
As $(M-I)^{p^m} \in \N^{2p^m}$,
one has $(M-I)^{p^m}=0$ when $2p^m > n$. The lemma follows.
\end{proof}

\begin{lem}\label{l:exponent}
If $2 \leq n \leq 6$, the outer exponent of $\U^1$ is equal to $p$.
\end{lem}
\begin{proof}
As the abelianization
of~$\U^1$ has exponent~$p$ when $n \geq 2$,
the outer exponent of~$\U^1$ is divisible by~$p$. Hence the desired result follows from Lemma~\ref{l:exponent-0} whenever $2p > n$. This holds for every $p > 3$ (as we assume $n \leq 6$), for $p=3$ when $n \leq 5$, and for $p=2$ when $n \leq 3$. We now address the remaining cases.

Suppose that $p=3$ and $n=6$.
By Lemma~\ref{l:exponent-0}, the exponent of $\U^1$ is divisible by $9$. Hence it suffices to show that $M^4$ is conjugate to $M$ for any $M \in \U^1$. As $(M-I)^3 \in \N^6$ and the $(0,6)$ entry
of this matrix is $M_{0,2}M_{2,4}M_{4,6}$, we have
\begin{align}
\label{eq:mmi3}
(M-I)^3 = M_{0,2}M_{2,4}M_{4,6}(e_{0,6}-I)\rlap.
\end{align}
Multiplying~\eqref{eq:newton} for $m=1$ on the left by~$M$ and using
the equality $R(e_{0,6}-I) = e_{0,6}-I$ for $R \in \U$,
we obtain, in view of~\eqref{eq:mmi3},
that
\begin{align}
M^4 = M + M(M-I)^3 = M + M_{0,2}M_{2,4}M_{4,6}(e_{0,6}-I)\rlap.
\end{align}
Now conjugation by~$e_{2,6}$ acts
on any matrix of~$\U^1$ by subtracting its $(0,2)$ entry from its $(0,6)$ entry. We conclude that 
\begin{align*}
QM^4Q^{-1} = M
\end{align*}
with $Q = e_{2,6}^{M_{2,4}M_{4,6}}$, as desired.

Finally, suppose that $p=2$ and $4 \leq n\leq 6$.
By Lemma~\ref{l:exponent-0},
the exponent of~$\U^1$ divides~$4$.  To prove that its outer exponent is~$2$,
it therefore suffices to show that each $M \in \U^1$ is conjugate to its inverse.
To this end,
define $Q \in \U^1$ by setting $Q_{i,j}=M_{i,j}$ if $i>3$ or $j\leq 3$,
and $Q_{i,j}=0$ otherwise, and let $R=MQ^{-1}$. 
We may depict
 $M$, $Q$ and $R$ 
as block matrices
\[ M = \begin{pmatrix}  A & \rvline & B \\ \hline  0 & \rvline & C \end{pmatrix} \quad\quad
Q = \begin{pmatrix}  A & \rvline & 0 \\ \hline  0 & \rvline & C \end{pmatrix} \quad\quad
R =\begin{pmatrix}  I & \rvline & D \\ \hline  0 & \rvline & I \end{pmatrix}\rlap, 
\]
where~$A$ has size $4\times 4$ and $C$ has size $(n-3) \times (n-3)$.
Lemma~\ref{l:exponent-0} applied with $p=2$ and $n\leq 3$ shows that $A^2=I$ and $C^2=I$,
so that $Q^2=I$. On the other hand, as $2D=0$, we also have $R^2 = I$, and so $MQ^{-1} = QM^{-1}$.
We conclude that
\begin{align*}
QM^{-1}Q^{-1}=QM^{-1}Q = M\rlap,
\end{align*}
as desired.
\end{proof}

We do not know whether the conclusion of Lemma~\ref{l:exponent} remains true for all $n$ and $p$.

\begin{rem}
\label{r:beforebzero}
Lemma~\ref{l:exponent} implies that when $n \leq 6$ the map $\alp\colon G \lrar \A$ 
is injective as soon as $k$ contains a primitive $p$-th root of unity (e.g., when $p=2$).
\end{rem}

When $n\geq 3$, we will denote by $\B_0 \subseteq \B$ the subgroup generated by $\ovl{e}_{0,2},\ovl{e}_{0,3},\ovl{e}_{n-3,n},\ovl{e}_{n-2,n}$. We note that $\B_0$ is closed under the action of $\A$ and that $\B_0$ is an elementary abelian $p$\nobreakdash-group of rank $4$, except when $n=3$ in which case $\ovl{e}_{0,3}$ and $\ovl{e}_{n-3,n}$ coincide and $\B_0 = \B$ has rank~$3$. 
The following result on the image of $\Br_{1,\nr}(V)$ in $H^1(G,\what{\B})$ will play a key role in the proof of the Massey vanishing conjecture when $n \leq 6$:

\begin{prop}\label{p:bound_on_brnr1}
If $3 \leq n \leq 6$, the subgroup $\Br_{1,\nr}(V)/\Br_0(V) \subseteq H^1(G,\what{\B})$ is contained in the kernel of the natural map
$H^1(G,\what{\B}) \lrar H^1(G,\what{\B}_0)$.
\end{prop}

The remainder of this section is devoted to the proof of Proposition~\ref{p:bound_on_brnr1}.
We henceforth assume that $n\geq 3$.

\begin{rem}\label{r:aid-memoire}
Let $\sig = \sum_{i=0}^{n-1} a_i\ovl{e}_{i,i+1} \in \A$. The action of $\sig$ on $\U^1/(\mathrm{conjugacy})$ can be computed by lifting $\sig$ to the unitriangular matrix $S$ such that $S_{i,i+1}=a_i$ and all other non-diagonal values vanish:
on the level of conjugacy representatives $Q \in \U^1$, the action of~$\sig$ is determined by $\sig([Q]) = [M]$ with $M=SQS^{-1}$.
 We have $(S^{-1})_{i,j} = (-1)^{j-i}\prod_{l=i}^{j-1}a_l$ when~$i\leq j$.  Performing the calculation and using the fact that $Q_{i,i}=1$ and $Q_{i,i+1}=0$, we find that $M$ is given by
\begin{align}
\begin{aligned}
\label{e:aide}
 M_{i,j} &= \left[\sum_{m=i}^j (-1)^{j-m}Q_{i,m}\prod_{l=m}^{j-1} a_l\right] + a_{i}\left[\sum_{m=i+1}^j(-1)^{j-m}Q_{i+1,m}\prod_{l=m}^{j-1} a_l\right]  \\
&= Q_{i,j} + \sum_{m=i}^{j-1} (-1)^{j-m}\left[\big(Q_{i,m}a_{m} - a_{i}Q_{i+1,m+1} \big)\prod_{l=m+1}^{j-1}a_l\right] \\
&= Q_{i,j} + \sum_{m=i+2}^{j-1} (-1)^{j-m}\left[\big(Q_{i,m}a_{m} - a_{i}Q_{i+1,m+1} \big)\prod_{l=m+1}^{j-1}a_l\right]\in \FF_p
\rlap{.}
\end{aligned}
\end{align}
Here, by convention, we interpret an empty sum as $0$ and an empty product as $1$. In particular, 
we have  $M_{i,j} = Q_{i,j}$ when $j-i \leq 2$ and $M_{i,i+3} = Q_{i,i+3}-Q_{i,i+2}a_{i+2} + a_{i}Q_{i+1,i+3}$.
\end{rem}

\begin{rem}\label{r:symmetry}
Let $\tau$ denote the involution of $\GL_{n+1}(\FF_p)$ defined by $\tau(M)=(PM^{-1}P^{-1})^t$,
where $P=P^{-1}$ is the matrix of the permutation $i \mapsto n-i$.
We note that $\tau(e_{i,j})=e_{n-j,n-i}^{-1}$
for all~$i$ and~$j$, so that $\tau(\U^i)=\U^i$ for all~$i$.  We therefore obtain induced involutions,
again denoted~$\tau$,
on~$\A$ and~$\B$. The action of~$\A$
on $\U^1/(\mathrm{conjugacy})$
respects~$\tau$ in the sense that
$\tau(\sigma)(\tau([Q]))=\tau(\sigma([Q]))$ 
for any $\sigma \in \A$ and any $Q \in \U^1$, and similarly for the action of $\A$ on $\B$. The observant reader might be troubled by the fact that the formula in~\eqref{e:aide} is not visibly symmetric with respect to $\tau$: if we replace $\sig$ by $\tau(\sig)$ and $Q$ by $\tau(Q)$ we will not get the matrix $\tau(M)$ on the nose. The new $M$ will however be conjugate to $\tau(M)$. This is simply a side effect of our (essentially arbitrary) choice of $S$ as a lift of $\sigma$, which was not done in a $\tau$-symmetric manner.
\end{rem}

Recall that for a set $I$ equipped with an action of $G$ and an element $\sig \in G$, we write $I^{\sig}$ for the subset of $\sig$-invariant elements of $I$.

\begin{lem}\label{l:b2}
Let $\sig \in \A$ be an element and let $b \in \B^{\sig}$ be a $\sig$-invariant element which is contained in the subgroup of $\B$ spanned by $\{\ovl{e}_{i,i+2}\}_{i=0}^{n-2}$. 
Then $b$ lifts to a $\sig$-invariant element of $\U^1/(\mathrm{conjugacy})$.
\end{lem}
\begin{proof}
Let us write $b = \sum_{i=0}^{n-2} b_{i}\ovl{e}_{i,i+2}$ and $\sig = \sum_{i=0}^{n-1} a_i\ovl{e}_{i,i+1}$. Since $b$ is $\sig$-invariant, we have by~\eqref{e:T} that $-b_{i}a_{i+2} + a_{i}b_{i+1} =0$ for every $i=0,\dots,n-3$. It then follows from~\eqref{e:aide} 
that if we lift $b$ to the unitriangular matrix $Q$ such that $Q_{i,i+2}=b_{i}$ and all other non-diagonal values vanish, then the conjugacy class $[Q] \in \U^1/(\mathrm{conjugacy})$ is $\sig$-invariant.
\end{proof}

\begin{lem}\label{l:group-theory}
For any $\sigma \in \A$, the subgroup of $\B^\sigma$ generated by the image of the natural map
\begin{equation}\label{e:conj}
(\U^1/(\mathrm{conjugacy}))^\sigma \lrar \B^\sigma
\end{equation}
contains $\B_0^{\sig}$. 
\end{lem}
\begin{proof}
If $n=3$, then $\U^1=\B$ and the statement is trivial. Let us assume that $n \geq 4$, in which case $\B_0$ is generated by four distinct elements $\ovl{e}_{0,2},\ovl{e}_{0,3},\ovl{e}_{n-3,n},\ovl{e}_{n-2,n}$. We first note that since every $\ovl{e}_{i,i+3} \in \B$ is $\A$-invariant, Lemma~\ref{l:b2} implies that any $\sig$-invariant element in $\B_0$ can be written as a sum of an element spanned by $\ovl{e}_{0,3},\ovl{e}_{n-3,n}$ and an element in the image of~\eqref{e:conj}. It will hence suffice to show that $\ovl{e}_{0,3}$ and $\ovl{e}_{n-3,n}$ can be written as sums of elements in the image of~\eqref{e:conj}. We will give the argument for $\ovl{e}_{0,3}$. The case of $\ovl{e}_{n-3,n}$ then follows from the symmetry of Remark~\ref{r:symmetry}. 

Let us write $\sig = \sum_i a_i \ovl{e}_{i,i+1}$. If $a_3=0$ then $[e_{0,3}] \in \U^1/(\mathrm{conjugacy})$ is a $\sig$-invariant lift of $\ovl{e}_{0,3}$ by the formula in~\eqref{e:aide}. If $a_3 \neq 0$ but $a_2=0$ then we write $\ovl{e}_{0,3}$ as the difference between $\ovl{e}_{0,3}+\ovl{e}_{0,2}$ and $\ovl{e}_{0,2}$, both of which are $\sig$-invariant in $\B$ since $a_2=0$. Let $Q = e_{0,3}e_{0,2}$ serve as a lift of $\ovl{e}_{0,3}+\ovl{e}_{0,2}$. By Remark~\ref{r:aid-memoire}, the conjugacy class $\sig([Q])$ can be represented by a matrix $M$ such that $M_{i,j} = Q_{i,j}$ when $j-i \leq 3$ and $M_{i,j} = 0$ when $j>i > 0$.
Then $Q$
is conjugate to $M$ via the element $R := \prod_{j=4}^{n}e_{2,j}^{-M_{0,j}}$ (that is, $RQR^{-1} = M$), so that $[Q] \in (\U^1/(\mathrm{conjugacy}))^\sig$. We conclude that $\ovl{e}_{0,3}+\ovl{e}_{0,2}$ belongs to the image
of~\eqref{e:conj}.
By Lemma~\ref{l:b2},
so does $\ovl{e}_{0,2}$.
Hence $\ovl{e}_{0,3}$ lies in the subgroup generated by the image of~\eqref{e:conj}. 

Finally, suppose that $a_2\neq 0$. The element $b = \sum_{i=0}^{n-2} b_i\ovl{e}_{i,i+2}\in \B$ defined by $b_i := a_ia_{i+1}$ is $\sig$\nobreakdash-invariant and we may write $\ovl{e}_{0,3}$ as a linear combination of $a_2\ovl{e}_{0,3}+b$ and $b$. As above, $b$ lifts to a $\sig$-invariant conjugacy class by Lemma~\ref{l:b2}. Let $Q \in \U^1$ be the matrix such that $Q_{i,i+2} = b_i, Q_{0,3} = a_2$ and all other non-diagonal values vanish, 
so that $Q$ is a lift of $a_2\ovl{e}_{0,3}+b$. Let $M$ be the matrix representing $\sig([Q])$ as in Remark~\ref{r:aid-memoire}.
Then $M_{i,j}=Q_{i,j}$ if $j-i \leq 3$, $M_{i,j} = 0$ if $j-3>i>0$ and 
\[ M_{0,j}=(-1)^{j-3}\prod_{l=2}^{j-1}a_l =  (-1)^{j-3}\left[\prod_{l=2}^{j-3}a_l\right] \cdot b_{j-2} \] 
for $j=4,\dots,n$. (Note that the terms associated with $m=i+2$ in Formula~\eqref{e:aide} vanish by the construction of $b$.)
We then see that $Q$ is conjugate to $M$ via the element $\prod_{j=4}^{n}e_{0,j-2}^{\eps_j}$, where $\eps_j=(-1)^{j-3}\prod_{l=2}^{j-3}a_l$. Indeed, we may write $Q$ as the product $Q = e_{n-2,n}^{b_{n-2}} \dots e_{0,2}^{b_0} e_{0,3}$ and $M$ as $Q\prod_{j=4}^ne_{0,j}^{M_{0,j}}$, while for $j \in \{4,\dots,n\}$ we have $e_{0,j-2}Qe_{0,j-2}^{-1} = Qe_{0,j}^{b_{j-2}}$.
In particular, the conjugacy class $[Q]$ is $\sig$-invariant. Thus $\ovl{e}_{0,3}$ again belongs to the subgroup generated by the image of~\eqref{e:conj}, as desired.
\end{proof}

\begin{lem}\label{l:group-theory2}
If $n \leq 6$, the image of $\Br_{1,\nr}(V)/\Br_0(V) \subseteq H^1(G,\what{\B})$ 
by the natural map
$H^1(G,\what{\B}) \lrar H^1(G,\what{\B}_0)$
is contained in $\Sha^1_{\cyc}(G,\what{\B}_{0})$.
\end{lem}
\begin{proof}
By Proposition~\ref{p:neukirch}, we may assume that $k$ contains a primitive $p$-th root of unity.
By Lemma~\ref{l:exponent}, the cyclotomic character $\chi\colon G \lrar (\ZZ/e)^* = (\ZZ/p)^*$ is then trivial. 
Let us fix an element $\beta \in H^1(G,\what{\B}_0)$ lying in the image of $\Br_{1,\nr}(V)/\Br_0(V)$. 
By Corollary~\ref{c:special-case}
and Lemma~\ref{l:group-theory}, we find that for every $\sig \in G$, 
if we let~$\ZZ$ act on~$B_0$ and on $\widehat{B}_0$
through~$\sigma$,
the pull-back $\sigma^* \beta \in H^1(\ZZ,\what{\B}_0)$ lies in the right kernel of the cup product pairing
\[ H^0(\ZZ,\B_0) \times H^1(\ZZ,\what{\B}_0) \lrar H^1(\ZZ,\ovl{\mmu}_{\infty}) = \ovl{\mmu}_{\infty} \rlap{.}\]
This pairing being perfect (Poincaré duality for the group $\ZZ$, which amounts to the duality between the $\sig$-invariants of $\B_0$ and the $\sig$-coinvariants of $\what{\B}_0$),
we get that $\sigma^*\beta\in H^1(\ZZ,\what{\B}_0)$ vanishes  
for all $\sigma \in G$. We conclude that the image of $\beta$ in $H^1(G,\what{B}_0)$ lies in $\Sha^1_{\cyc}(G,\what{\B}_0)$, as desired.
\end{proof}

\begin{lem}\label{l:sha}
One has $\Sha^1_{\cyc}(G,\what{\B}_{0})=0$.
\end{lem}
\begin{proof}
Let $H \subseteq G$ be the kernel of the composed map $G \x{\chi}{\lrar} (\ZZ/e)^* \lrar (\ZZ/p)^*$.
The restriction map $H^1(G,\what{\B}_0) \lrar H^1(H,\what{\B}_0)$ is injective
since the index of~$H$ in~$G$ is prime to~$p$ while $\what{\B}_0$ has exponent~$p$.
It will hence suffice to show that $\Sha^1_{\cyc}(H,\what{\B}_0) =0$.

Assume first that $n > 3$. In this case ${\B}_0$ splits as a direct sum of two $H$-modules ${\B}_0 = {\B}_0'\oplus {\B}_0''$ where ${\B}_0'$ is spanned by $\ovl{e}_{0,2},\ovl{e}_{0,3}$ and ${\B}_0''$ is spanned by $\ovl{e}_{n-3,n},\ovl{e}_{n-2,n}$. Furthermore, if we let $f_i\colon \A \lrar \FF_p$ denote the homomorphism such that $f_i(\ovl{e}_{j,j+1}) = \del_{i,j}$ (Kronecker's delta) then the action of $H$ on ${\B}_0'$ factors through the composed map $H \lrar \A \x{f_2}{\lrar} \FF_p$ and the action of $H$ on ${\B}_0''$ factors through the composed map $H \lrar \A \x{f_{n-3}}{\lrar} \FF_p$. In particular, $\B_0$ is a direct sum of two $H$\nobreakdash-modules
on each of which the action of~$H$ factors through a cyclic quotient. The same must then be true for the dual module $\what{\B}_0$, and so $\Sha^1_{\cyc}(H,\what{\B}_0) =0$.
 
Let us now consider the case where $n=3$.
By dualising the short exact sequence
\begin{align}
0 \lrar \FF_p \lrar \B_0 \lrar \FF_p^2 \lrar 0\rlap{,}
\end{align}
where the first map sends~$1$ to $\ovl{e}_{0,3}$ and the quotient
$\FF_p^2=\B/\< \ovl{e}_{0,3}\>$
is generated by the images of $\ovl{e}_{0,2},\ovl{e}_{1,3}$,
we obtain an exact sequence of $H$\nobreakdash-modules
\begin{align}
\label{eq:dualb}
 0 \lrar \FF_p^2 \lrar \what{\B}_0 \lrar \FF_p \lrar 0\rlap{.}
\end{align}
As $\Sha^1_{\cyc}(H,\FF_p)=0$, any element in $\Sha^1_{\cyc}(H,\what{\B}_0)$ 
comes from $H^1(H,\FF_p^2)=\Hom(H,\FF_p^2)$.
Let $T_0 \in \Hom(H,\FF_p^2)$ be the image
of $1 \in H^0(H,\FF_p)$ by
the boundary of~\eqref{eq:dualb}. We are now reduced to verifying
the following easy fact from linear algebra:
if $T \in \Hom(H,\FF_p^2)$
is such that $T(h)$ is a multiple of~$T_0(h)$ for
any $h \in H$, then~$T$ is itself a multiple of $T_0$.
\end{proof}

\begin{proof}[Proof of Proposition~\ref{p:bound_on_brnr1}]
Combine Lemma~\ref{l:group-theory2} and Lemma~\ref{l:sha}.
\end{proof}

\begin{cor}\label{c:n3}
If $n=3$, then $\Br_{\nr}(V) = \Br_0(V)$. 
\end{cor}
\begin{proof}
When $n=3$, the inclusion $\B_0 \subseteq \B$ is an equality and hence Proposition~\ref{p:bound_on_brnr1} and Proposition~\ref{p:trans} imply the vanishing of $\Br_{\nr}(V)/\Br_0(V)$. 
\end{proof}

\begin{rem}
With more care and more computations, the same ideas lead to a complete determination of the unramified Brauer group of~$V$ also in the cases $n=4,5,6$ (see \S\ref{s:brauer-2}). In particular, when $k$ is a number field and $n=4,5$, we find that 
\[ \Br_{\nr}(V)/\Br_0(V) = \Ba_{\omega}(V) = \Sha^1_{\cyc}(G,\what{\B}) \]
reduces to the group of Brauer elements which are locally constant almost everywhere (modulo constant classes), see Remark~\ref{r:sha-cyc-is-ba}.
We note that this group can be non-trivial (see Example~\ref{e:n4}). When $n=6$, the inclusion $\Ba_\omega(V) \subseteq \Br_{\nr}(V)/\Br_0(V)$ can fail to be an equality.
\end{rem}

\section{Proof of the Massey vanishing conjecture}\label{s:main}
Our goal in this section is to prove the main theorem of this paper:

\begin{thm}\label{t:massey}
The Massey vanishing conjecture holds for every number field $k$, every $n \geq 3$ and every prime $p$.
\end{thm}

Let us henceforth fix a number field $k$, a prime $p$ and an integer $n \geq 3$.
\begin{define}
For $1 \leq r,s \leq n-2$ let us denote by $\P^{r,s} \subseteq \U^1 \subseteq \U$ the subgroup of $\U^1$ generated by the elementary matrices $\{e_{i,n}\}_{i=0}^{r}$ and $\{e_{0,j}\}_{j=n-s}^{n}$.
\end{define}

The matrices in the subgroup $\P^{r,s}$ can be visually depicted as:
\NiceMatrixOptions{code-for-first-row = \color{darkblue}, code-for-last-col = \color{darkblue}}
\[
\begin{pNiceMatrix}[first-row,last-col]
 &   &   &     &   &   &     &   & n-s & \Cdots    &  &  n &  \\
1 & & \color{gray}{0} & \Cdots & \color{gray}{0} & \color{gray}{0} & \Cdots & \color{gray}{0} & \color{darkgreen}{*} & \Cdots & \color{darkgreen}{*} & \color{darkgreen}{*} & 0 \\
\color{gray}{0} &  & 1 &     & \color{gray}{0} & \color{gray}{0} &     & \color{gray}{0} & \color{gray}{0} &     & \color{gray}{0} & \color{darkgreen}{*} &  \smash{\Vdots} \\
\Vdots &&&\ddots &&&&&&&& \Vdots & \\
\color{gray}{0} &  & \color{gray}{0} &     & 1 & \color{gray}{0} &     & \color{gray}{0} & \color{gray}{0} &     & \color{gray}{0} & \color{darkgreen}{*} & r \\
\color{gray}{0} &  & \color{gray}{0} &     & \color{gray}{0} & 1 &     & \color{gray}{0} & \color{gray}{0} &     & \color{gray}{0} & \color{gray}{0} &  \\
\Vdots &&&&&& \ddots &&&&& \Vdots & \\
\color{gray}{0} &  & \color{gray}{0} &     & \color{gray}{0} & \color{gray}{0} &     & 1 & \color{gray}{0} &     & \color{gray}{0} & \color{gray}{0} &  \\
\color{gray}{0} &  & \color{gray}{0} &     & \color{gray}{0} & \color{gray}{0} &     & \color{gray}{0} & 1 &     & \color{gray}{0} & \color{gray}{0} &  \\
\Vdots &&&&&&&&& \ddots && \Vdots & \\
\color{gray}{0} &  & \color{gray}{0} &     & \color{gray}{0} & \color{gray}{0} &     & \color{gray}{0} & \color{gray}{0} &     & 1  & \color{gray}{0} &  \\
\color{gray}{0} &  & \color{gray}{0} & \Cdots  & \color{gray}{0} & \color{gray}{0} & \Cdots & \color{gray}{0} & \color{gray}{0} & \Cdots & \color{gray}{0} & 1 &  \\
\end{pNiceMatrix}
\]

We then observe that:
\begin{enumerate}
\item
The group $\P^{r,s}$ is normal in $\U$ and contains the center $\Z$.
\item
The quotient $\P^{r,s}/\Z$ is an elementary abelian $p$-group of rank $r+s$.
\item\label{i:3}
If $r+s \leq n-1$ (so that $r < n-s$) then $\P^{r,s}$ itself is abelian and the short exact sequence of abelian groups
\begin{align}
\begin{aligned}
\label{e:P}
\xymatrix@R=3ex{
1 \ar[r] & \Z \ar[r] & \P^{r,s} \ar[r] & \P^{r,s}/\Z \ar[r]& 1
}
\end{aligned}
\end{align}
splits.
\end{enumerate}
Using Proposition~\ref{p:splitting} and Proposition~\ref{p:fibration}, the proof of Theorem~\ref{t:massey} will eventually rely on constructing local Massey solutions satisfying certain constraints. Our main tool for constructing such local solutions is the following proposition.
Here and below, by a \emph{local field}, we mean a complete discretely valued field with finite residue field.

\begin{prop}\label{p:local}
Let $K$ be a local field. Let $n \geq 3$ and $1 \leq r,s \leq n-2$.
If $r+s=n-1$, then any homomorphism $\Gam_K \lrar \U/\P^{r,s}$ that lifts to a homomorphism $\Gam_K \lrar \U/\Z$ also lifts to a homomorphism $\Gam_K \lrar \U$.
\end{prop}

\begin{rem}
In the situation of Proposition~\ref{p:local}, if we were to replace the subgroup $\P^{r,s} \subseteq \U$ by the subgroup $\U^1 \subseteq \U$, then we would obtain the statement that $n$-fold Massey products in $K$ vanish as soon as they are defined (a statement which is well known to hold, see \cite[Theorem 4.3]{MN17a}). Since $\Z \subseteq \P^{r,s} \subseteq \U^1$ we may consider Proposition~\ref{p:local} as a \emph{refinement} of the Massey vanishing property for $K$. 
\end{rem}

The proof of Proposition~\ref{p:local} will require the following homological algebra lemma.

\begin{lem}\label{l:boundary}
Let $\Gam$ be a profinite group and
\begin{equation}\label{eq:short-exact-sequence}
0 \to A \xrightarrow{\iota} B \xrightarrow{\kappa} C \to 0 
\end{equation}
be a short exact sequence of discrete $\Gam$-modules
whose underlying sequence of abelian groups is split.
The
set of homomorphisms $C \to B$ that are sections of~$\kappa$
is
equipped with a natural action of~$\Gam$ by conjugation
and  then
forms a
torsor under the $\Gam$-module $\Hom(C,A)$.
Let $[\gam] \in H^1(\Gam,\Hom(C,A))$ be the class of this torsor. Then for any $m \geq 0$, the boundary map
\( H^m(\Gam,C) \to H^{m+1}(\Gam,A) \)
is given by cup product with $[\gam]$ (with respect to the canonical pairing $C \otimes \Hom(C,A) \to A$).
\end{lem}

\begin{proof}
Identifying $H^{\bullet}(\Gam,C) = \Ext^{\bullet}_{\Gam}(\ZZ,C)$ the boundary map \(H^m(\Gam,C) \to H^{m+1}(\Gam,A)\) is given, almost by definition, by the composition pairing with the element \(\eta \in \Ext_{\Gam}(C,A)\) classifying the short exact sequence~\eqref{eq:short-exact-sequence}. The result now follows from the compatibility of composition and cup pairings, see 
\cite[Proposition~0.14~(a)]{milneadtnew} (with $M=\Hom(C,A)$, $N=C$ and $P=A$).
\end{proof}

\begin{proof}[Proof of Proposition~\ref{p:local}]
If $K$ is of characteristic $p$ or $K$ does not contain a primitive $p$-th root of unity then $H^2(K,\FF_p) = 0$ (see~\cite[Theorem 9.1]{Koc13} in the former case and use local Tate duality for the latter), in which case the lemma trivially holds: any homomorphism $\Gam_K \lrar \U/\Z$ lifts to $\Z$ since the obstruction lives in $H^2(K,\Z)=H^2(K,\FF_p)$. We may hence assume that $K$ is of characteristic $\neq p$ and contains a primitive $p$-th root of unity.

Let us fix a homomorphism $\ovl{\alp}\colon\Gam_K \lrar \U/\P^{r,s}$
and a homomorphism $\wtl{\alp}\colon \Gam_K \lrar \U/\Z$ which lifts it. 
Then the obstruction to solving the lifting problem
\begin{align}
\begin{aligned}
\label{e:embedding-U-Z}
\xymatrix@R=3ex{
& & & \Gam_K \ar^{\wtl{\alp}}[d] \ar@{-->}[dl] & \\
1 \ar[r] & \Z \ar[r] & \U \ar[r] & \U/\Z \ar[r] & 1 \\
}
\end{aligned}
\end{align}
is encoded by the pull-back $\wtl{\alp}^*\eta \in H^2(K,\Z)$ along $\wtl{\alp}$ of the element $\eta \in H^2(\U/\Z,\Z)$ classifying the central extension of the bottom row in~\eqref{e:embedding-U-Z}. A 2-cocycle representing $\wtl{\alp}^*\eta$ can be obtained by choosing a 1-cochain $\hat{\alp}\colon \Gam_K \to \U$ lifting $\wtl{\alp}$, in which case $\wtl{\alp}^*\eta = [\phi]$ for the 2-cocycle $\phi_{\sig,\tau} := \hat{\alp}_{\sig}\hat{\alp}_{\tau}\hat{\alp}_{\sig\tau}^{-1}$.

Since $\P^{r,s}$ and $\Z$ are normal in $\U$ and $\Z$ is central, the conjugation action of~$\U$ on itself induces a compatible action of $\U/\Z$ on all the groups appearing in~\eqref{e:P}, and we subsequently consider it as a short exact sequence of Galois modules
by pulling back
this action 
along $\wtl{\alp}$.  
We may then consider the resulting boundary map
\begin{align}
\label{eq:boundaryfirst}
\partial\colon H^1(K,\P^{r,s}/\Z) \lrar H^2(K,\Z)\rlap{,}
\end{align}
where we point out that the Galois action on~$\Z$ is trivial, since so is the conjugation action of~$\U$ on~$\Z$.
According to~\eqref{i:3}, the sequence~\eqref{e:P} splits as a short exact sequence of abelian groups.
We now separately discuss the case where it splits $\Gam_K$-equivariantly and the case
where it does not.

Consider first the case where~\eqref{e:P} does not split $\Gam_K$-equivariantly.
Then the collection of (non-equivariant) sections  
$\P^{r,s}/\Z \to \P^{r,s}$ 
forms a non-trivial torsor under the Galois module $\Hom(\P^{r,s}/\Z,\Z)$, classified by some non-zero $[\gam] \in H^1(K,\Hom(\P^{r,s}/\Z,\Z))$, and by Lemma~\ref{l:boundary} the boundary map~\eqref{eq:boundaryfirst} is given by taking the cup product with $[\gam]$. As $[\gam]\neq 0$, as $\P^{r,s}$ is a $p$-torsion module and as
$Z \simeq \mmu_p$,
local duality for finite abelian Galois modules implies that this boundary map is surjective.
It follows that there exists a $1$-cocycle $\beta\colon \Gam_K \lrar \P^{r,s}/\Z$ such that $\partial[\beta] = -\wtl{\alp}^*\eta \in H^2(K,\Z)$. We note that the class $\partial[\beta]$ can be represented explicitly by choosing a 1-cochain $\hat{\beta}\colon \Gam_k \to \P^{r,s}$ lifting $\beta$, in which case $\partial[\beta] = [\psi]$ for the 2-cocycle $\psi_{\sig,\tau} := \hat{\beta}_{\sig}(\hat{\beta}_{\tau})^{\sig}\hat{\beta}_{\sig\tau}^{-1}$, where $(\hat{\beta}_{\tau})^{\sig}$ denotes the result of the action of $\sig$ on~$\hat{\beta}_{\tau}$. Since $-(\wtl{\alp}^*\eta + \partial[\beta]) = 0 \in H^2(K,\Z)$ there exists a 1-cochain $\xi\colon \Gam_K \to \Z$ such that $\xi_{\sig}\xi_{\tau}\xi_{\sig\tau}^{-1} = (\phi_{\sig,\tau}\psi_{\sig,\tau})^{-1}$. 
We now claim that the map
 $\hat{\alp}'\colon \Gam_K \lrar \U$
defined by the pointwise product
$ \hat{\alp}'_{\sig} = \hat{\bet}_{\sig}\hat{\alp}_{\sig}\xi_{\sig}$
is a group homomorphism. Indeed
\begin{align*}
\hat{\alp}'_{\sig}\hat{\alp}'_{\tau} & =
\hat{\beta}_{\sig}\hat{\alp}_{\sig}\xi_{\sig}\hat{\beta}_{\tau}\hat{\alp}_{\tau}\xi_{\tau} =
\hat{\beta}_{\sig}\hat{\alp}_{\sig}\hat{\beta}_{\tau}\hat{\alp}_{\tau}\xi_{\sig}\xi_{\tau} =
\hat{\beta}_{\sig}(\hat{\beta}_{\tau})^{\sig}\hat{\alp}_{\sig}\hat{\alp}_{\tau}\xi_{\sig}\xi_{\tau} \\
& = \psi_{\sig,\tau}\hat{\beta}_{\sig\tau}\phi_{\sig,\tau}\hat{\alp}_{\sig\tau}\xi_{\sig}\xi_{\tau} = \hat{\beta}_{\sig\tau}\hat{\alp}_{\sig\tau}\xi_{\sig\tau} = \hat{\alp}'_{\sig\tau} \rlap.
\end{align*}
We have thus found a lift of $\ovl{\alp}$ to a homomorphism $\Gam_K \lrar \U$. 

We now consider the case where~\eqref{e:P} splits $\Gam_K$-equivariantly. We will show that in this case $\wtl{\alp}$ itself lifts to $\U$.
To this end, let us identify $\Z$ with $\FF_p$ via the generator $e_{0,n} \in \Z$.
Let $W = \FF_p^{n+1}$ be the vector space of $(n+1)$-dimensional ``column vectors'', so that we have a natural left action of the ring of square matrices $\Mat_{n+1}(\FF_p)$ on $W$ via matrix multiplication on the left. Let $W^* = \Hom(W,\FF_p)$ be the dual space of $W$, which we can identify with the vector space of ``row vectors'' via the natural scalar product of row vectors and column vectors. In particular, we have a right action of $\Mat_{n+1}(\FF_p)$ on~$W^*$ via matrix multiplication on the right. Given $v \in W,u \in W^*$ and $M \in \Mat_{n+1}(\FF_p)$ we may then consider the associated scalar $uMv \in \FF_p$, obtained by applying the functional $u$ to the vector $Mv$, or, equivalently, the functional $uM$ to the vector $v$.
Given $v \in W, u \in W^*$ we may then consider the homomorphism
\begin{equation}\label{e:rho} 
\rho_{u,v}\colon \P^{r,s} \lrar \FF_p = \Z\text{,} \quad\quad\quad \rho_{u,v}(Q) = u(Q-I)v, 
\end{equation}
where $I \in \Mat_{n+1}(\FF_p)$ denotes the identity matrix. The map $\rho_{u,v}$ is indeed a group homomorphism since
\[ u(Q_1Q_2-I)v = u(Q_1-I)v + u(Q_2-I)v + u(Q_1-I)(Q_2-I)v \]
and $(Q_1-I)(Q_2-I) = 0$ for every $Q_1,Q_2 \in \P^{r,s}$ when $r+s=n-1$. Now, for $m\in \{0,\dots,n\}$, let $W^m \subseteq W$ be the $(m+1)$-dimensional subspace consisting of those column vectors whose last (bottom) $n-m$ coordinates vanish, and 
let $W^*_m \subseteq W^*$ denote the $(m+1)$-dimensional subspace consisting of those row vectors whose first (left) $n-m$ coordinates vanish. We note that the left action of $\U$ on $W$ preserves the filtration $W^0 \subseteq  \cdots \subseteq W^n=W$ and the right action of $\U$ on $W^*$ preserves the filtration $W^*_0 \subseteq \cdots \subseteq W^*_{n}$. We also note that the subgroup $\P^{r,s} \subseteq \U$ acts trivially on $W^*_{n-r-1}$ and $W^{n-s-1}$.
Since $\rho_{u,v}(e_{i,n}) = u_iv_n$ for $i=0,\dots,r$ and $\rho_{u,v}(e_{0,j}) = u_0v_j$ for $j=n-s,\dots,n$, we conclude:
\begin{enumerate}[label=(\roman*)]
\item\label{i:i}
The homomorphism $\rho_{u,v}$ is a retraction of $\Z \subseteq \P^{r,s}$ if and only if $u_0v_n=1$.
\item\label{i:ii}
Every retraction $\P^{r,s} \lrar \Z$ can be written as $\rho_{u,v}$ for some $u \in W^*, v \in W$.
\item\label{i:iii}
For $u,u' \in W^*$ and $v,v'\in W$ such that $u_0=u_0' \neq 0$ and $v_n=v_n' \neq 0$, we have $\rho_{u,v} = \rho_{u',v'}$ if and only if $u=u'$ mod $W^*_{n-r-1}$ and $v=v'$ mod $W^{n-s-1}$.
\end{enumerate}
Now by the assumption that~\eqref{e:P} splits equivariantly and by~\ref{i:i}--\ref{i:ii} above, there exist $u\in W^*$ and $v \in W$, with $u_0v_n=1$, such that the retraction $\rho_{u,v}\colon \P^{r,s} \lrar \Z$ is $\Gam_K$-equivariant. Let $\mS \subseteq \U$ be the subgroup consisting of those matrices $M \in \U$ such that $\rho_{u,v}(MQM^{-1}) = \rho_{u,v}(Q)$ for every $Q \in \P^{r,s}$,
i.e., such that $\rho_{uM,M^{-1}v}=\rho_{u,v}$.
As $\P^{r,s}$ is abelian we see that $\mS$ contains $\P^{r,s}$ and as $\rho_{v,u}$ is $\Gam_K$-equivariant, the subgroup $\mS/\Z \subseteq \U/\Z$ contains the image of $\wtl{\alp}\colon \Gam_K \lrar \U/\Z$. To finish the proof it will hence suffice to show that the map $\mS \lrar \mS/\Z$ admits a section, which, since $\Z$ is central, is equivalent to the assertion that the inclusion $\Z \subseteq \mS$ admits a retraction. In fact, we will show that the retraction $\rho_{u,v}\colon\P^{r,s} \lrar \Z$ itself extends to $\mS$. 

Applying~\ref{i:iii} to $(u',v')=(uM,M^{-1}v)$ and noting that $M W^{n-s-1} = W^{n-s-1}$, we see that
\[
M \in \mS \quad\Longleftrightarrow\quad u(M-I) \in W^*_{n-r-1} \quad\text{and}\quad (M-I)v \in W^{n-s-1}\rlap{.}
\]
As $(r+1)+(s+1)=n+1$, it follows that
$u(M_1-I)(M_2-I)v=0$
for any $M_1,M_2 \in S$.
As $u(M_1M_2-I)v = u(M_1-I)v + u(M_2-I)v + u(M_1-I)(M_2-I)v$,
we conclude that the formula $M \mapsto u(M-I)v$ provides an extension of $\rho_{v,u}$ to a homomorphism $\mS \lrar \Z$ yielding a retraction of the inclusion $\Z \subseteq \mS$, as desired.
\end{proof}

\begin{proof}[Proof of Theorem~\ref{t:massey}]
Let $\alp\colon \Gam_k \lrar \A$ be a homomorphism which lifts to a homomorphism $\wtl{\alp}\colon\Gam_k \lrar \U/\Z$. Using Dwyer's formulation as summarized in Corollary~\ref{c:dwyer}(3), what we need to show is that $\alp$ lifts to a homomorphism $\Gam_k \lrar \U$. 

Let $V = (\SL_N \times T_\alp)/\U$ be the homogeneous space of $\SL_N$ (with $N=|\U|+1$), over~$k$, associated with~$\alp$ by the construction of Pál--Schlank (see \S\ref{s:dwyerpalschlank}
and in particular~\eqref{eq:defv}). By Proposition~\ref{p:splitting}, our problem is equivalent to showing that $V$ has a rational point.
By Proposition~\ref{p:fibration}, which rests on~\cite[Théorème~B]{HW18},
it will suffice to show that $V$ contains a collection of local points orthogonal to $\Br_{\nr}(V)$
with respect to the Brauer--Manin pairing.

Let $1 \leq r,s \leq n-2$ be such that $r+s=n-1$ and let $\ovl{\alp}\colon \Gam_k \lrar \U/\P^{r,s}$ be the composition of $\wtl{\alp}$ with the projection $\U/\Z \lrar \U/\P^{r,s}$. When $n \geq 7$ we also impose the condition that $r,s \geq 3$. We may now repeat the construction of Pál--Schlank, this time for the embedding problem
\begin{align}
\begin{aligned}
\label{e:embedding-4}
\xymatrix@R=3ex{
& & & \Gam_k \ar^{\ovl{\alp}}[d] \ar@{-->}[dl] & \\
1 \ar[r] & \P^{r,s} \ar[r] & \U \ar[r] & \U/\P^{r,s} \ar[r] & 1 .\\
}
\end{aligned}
\end{align}
More precisely, let $T_{\ovl{\alp}} \lrar \spec(k)$ be the $k$-torsor under $\U/\P^{r,s}$ determined by
the homomorphism $\ovl{\alp}\colon \Gam_k \lrar \U/\P^{r,s}$ viewed as a $1$\nobreakdash-cocycle,
and let $W := (\SL_{N} \times T_{\ovl{\alp}})/\U$ be the quotient variety of $\SL_{N} \times T_{\ovl{\alp}}$ under the diagonal action of $\U$ on the right.
The left action of $\SL_N$ on the first factor of the product $\SL_N \times T_{\ovl{\alp}}$ descends uniquely to $W$, exhibiting it as a homogeneous space of $\SL_N$ with geometric stabilizer $\P^{r,s}$. In addition, the natural map $\SL_N \times T_{\ovl{\alp}} \lrar \SL_N \times T_{\alp}$ descends to an $\SL_N$-equivariant map $\pi\colon W \lrar V$ which realizes over~$k$ the covering map  $\SL_{N,\ovl{k}}/\P^{r,s} \lrar \SL_{N,\ovl{k}}/\U^1 \cong V_{\ovl{k}}$.

When $n \geq 7$, the conditions $r,s \geq 3$ and $r+s=n-1$ imply that $n-r,n-s \geq 4$ and hence that $\P^{r,s}$ is contained in $\U^3 = \Ker[\U^1 \lrar \B]$. It follows that the horizontal maps in the commutative diagram  
\begin{align}
\begin{aligned}
\label{eq:brvw}
\xymatrix{
\Br_{1,\nr}(V)/\Br_0(V)\ar@{^{(}->}[d]\ar[r] & \Br_{1,\nr}(W)/\Br_0(V) \ar@{^{(}->}[d] \\
H^1(k,\what{\B}) \ar[r] & H^1(k,\what{\P}^{r,s})
}
\end{aligned}
\end{align}
vanish.
When $n \leq 6$, the image of $\P^{r,s}$ in $\B$ is at least contained in the subgroup $\B_0 \subseteq \B$ introduced
after Remark~\ref{r:beforebzero},
and hence, by Proposition~\ref{p:bound_on_brnr1}, the top horizontal map in~\eqref{eq:brvw} still vanishes.
Thus, in any case,
the pull-back of any algebraic unramified  
Brauer class on $V$ (and hence any unramified Brauer class by Proposition~\ref{p:trans}) 
becomes constant in $\Br(W)$. 
In view of the projection formula, it follows that for any $(y_v) \in \prod_v W(k_v)$,
the family $(\pi(y_v)) \in \prod_v V(k_v)$ is orthogonal to $\Br_{\nr}(V)$.
We are thus reduced to checking that $W(k_v)\neq\emptyset$ for any place~$v$ of~$k$.
By~\cite[Theorem 9.6]{PS16}, this is equivalent to the embedding problem~\eqref{e:embedding-4} being locally solvable, which is exactly what Proposition~\ref{p:local} provides when applied to the restriction of~$\ovl{\alp}$ to each $\Gam_{k_v} \subseteq \Gam$. 
\end{proof}

\section{More Brauer group computations}\label{s:brauer-2}

In this section we give more elaborate computations of the unramified Brauer group of our splitting variety $V$ when $n=4,5$, which may be interesting in their own right, but are not strictly needed for the proof of the main theorem.
\begin{prop}\label{p:n45}
Suppose that $n=4,5$ and let $\sig \in \A$ be an element. Then the subgroup of $\B^{\sigma}$ generated by the image of the map
\begin{equation}\label{e:conj-2}
(\U^1/(\mathrm{conjugacy}))^\sigma \lrar \B^\sigma
\end{equation} 
is all of $\B^{\sigma}$.
\end{prop}
\begin{proof}
When $n=4$ the subgroup $\B_0 \cap \B^\sigma$ contains all $\ovl{e}_{i,j}$ with $j-i=3$ and hence the combination of Lemma~\ref{l:group-theory} and Lemma~\ref{l:b2} implies that for every $\sig \in \A$ the image of~\eqref{e:conj-2} generates $\B^{\sigma}$. 

Let us now consider the case $n=5$. In this case the subgroup $\B_0 \subseteq \B$ contains $\ovl{e}_{0,3}$ and $\ovl{e}_{2,5}$, which are invariant under~$\sigma$, and so by Lemma~\ref{l:group-theory} and Lemma~\ref{l:b2} it will suffice to show that $\ovl{e}_{1,4}$ is in the image of~\eqref{e:conj-2}.
Let us write $\sig = \sum_i a_i \ovl{e}_{i,i+1}$. If $a_4 = 0$ then the subgroup $\P_4 \subseteq \U^1$ spanned by those $e_{i,j}$ with $j \leq 4$ is stable under the outer action of $\sig$. Of course, $\P_4$ is just a copy of the group of unitriangular $4 \times 4$-matrices whose first non-principal diagonal vanishes and so the $n=4$ case of the desired claim implies that $\ovl{e}_{1,4}$ is contained in the subgroup generated by the image of~\eqref{e:conj-2}. Similarly, if $a_0 = 0$ then we may run the same argument using the subgroup generated by those $e_{i,j}$ such that $i \geq 1$. We may hence assume without loss of generality that $a_0,a_4 \neq 0$. We now claim that in this case there must exist $b_0,\dots,b_3\in\FF_p$ such that $b = \sum_i b_i\ovl{e}_{i,i+2}$ is $\sig$-invariant and such that either $b_0b_1 \neq 0$, or $b_2b_3 \neq 0$, or $b_0b_3 \neq 0$. Indeed, one of the following cases must hold,
in each of which the indicated~$b$ can be checked to be $\sig$\nobreakdash-invariant with the aid of~\eqref{e:T}:
\begin{enumerate}
\item
If $a_2 = 0$ then $\ovl{e}_{0,2}$ and $\ovl{e}_{3,5}$ are $\sig$-invariant and so we can take $b = \ovl{e}_{0,2} + \ovl{e}_{3,5}$.
\item
If $a_2 \neq 0$ and $a_3 = 0$ then we can take $b = a_0\ovl{e}_{0,2} + a_2\ovl{e}_{1,3}$. 
\item
If $a_2 \neq 0$ and $a_1 = 0$ then we can take $b = a_4\ovl{e}_{3,5} + a_2\ovl{e}_{2,4}$. 
\item
If $a_1,a_2,a_3 \neq 0$ then $b_i := a_ia_{i+1}$ is non-zero for every $i=0,\dots,3$. In this case $b = \sum_i b_i\ovl{e}_{i,i+2}$ is $\sig$-invariant and satisfies the required property.
\end{enumerate}
Now given a $\sig$-invariant $b = \sum_i b_i\ovl{e}_{i,i+2}$ as above we write $\ovl{e}_{1,4}$ as the difference between $\ovl{e}_{1,4}+b$ and $b$. By Lemma~\ref{l:b2} the element $b$ lifts to a $\sig$-invariant conjugacy class and hence it will suffice to show that $\ovl{e}_{1,4}+b$ lifts to a $\sig$-invariant conjugacy class. In fact, we will show that $\ovl{e}_{1,4}+b$ lifts to a \emph{unique}, and therefore $\sig$-invariant, conjugacy class.

Let
$Q \in \U^1$ be the matrix with $Q_{i,i+2}=b_i$ for all~$i$, $Q_{1,4}=1$ and all other non-diagonal entries~$0$,
so that $Q$ maps to $\ovl{e}_{1,4}+b$ in $\B$.
We recall that $\B$ is the abelianization of $\U^1$ and when $n=5$ the derived subgroup $\U^3 = \Ker[\U^1 \lrar \B]$ is contained in the center of $\U^1$. It follows that the conjugation action of $\U^1$ on itself is by automorphisms which induce the identity on both $\B$ and $\U^3$. The group of such automorphisms is naturally isomorphic to $\Hom(\B,\U^3)$, where an automorphism $\phi\colon\U^1\lrar \U^1$ and the corresponding
homomorphism $\psi\colon\B\lrar \U^3$ are related by $\phi(u)=u\psi(\bar u)$ and $\psi(\bar u)=u^{-1}\phi(u)$
for $u \in \U^1$.
 In particular, we may associate with $Q$ the element $\Theta_Q \in \Hom(\B,\U^3)$ which corresponds to the automorphism of conjugation by $Q$. 
More explicitly, $\Theta_Q\colon\B \lrar \U^3$ is given by $\Theta_Q(\ovl{M}) = M^{-1}QMQ^{-1} \in \U^3$ for any lift $M \in \U^1$ of $\ovl{M}$.
As $\U^3$ is generated by $e_{0,4},e_{1,5}$ and $e_{0,5}$, as
\begin{align*}
\Theta_Q(\ovl{e}_{2,4}) &= e_{0,4}^{b_0}\rlap{,}&
\Theta_Q(\ovl{e}_{2,5}) &= e_{0,5}^{b_0}\rlap{,}&
\Theta_Q(\ovl{e}_{3,5}) &= e_{1,5}^{b_1}\rlap{,}\\
\Theta_Q(\ovl{e}_{0,2}) &= e_{0,4}^{-b_2}\rlap{,}&
\Theta_Q(\ovl{e}_{0,3}) &= e_{0,5}^{-b_3}\rlap{,}&
\Theta_Q(\ovl{e}_{1,3}) &= e_{1,5}^{-b_3}\rlap{,}
\end{align*}
and as at least one of $b_0b_1$, $b_2b_3$, $b_0b_3$ is non-zero,
the map $\Theta_Q$ is surjective.
In particular, any lift of $\ovl{e}_{1,4}+b$ to~$\U^1$ can be written as $\Theta_Q(\ovl{M})Q$ for some $M \in \U^1$
and therefore belongs to the conjugacy class of~$Q$, as desired.
\end{proof}

\begin{cor}\label{c:n45}
If~$k$ is a number field and $n=4,5$, then $\Br_{\nr}(V)/\Br_0(V)$ coincides with the subgroup $\Ba_{\omega}(V)$ of Brauer elements which are locally constant almost everywhere (see Remark~\ref{r:sha-cyc-is-ba}).
\end{cor}
\begin{proof}
If $k$ does not contain a primitive $p$-th root of unity then $\Br_{\nr}(V) = \Br_0(V)$, by Proposition~\ref{p:neukirch} and Proposition~\ref{p:trans}. Assume now that $k$ contains a primitive $p$-th root of unity. By Lemma~\ref{l:exponent}, the cyclotomic character $\chi\colon G \lrar (\ZZ/e)^* = (\ZZ/p)^*$ is then trivial. We may hence conclude from
Proposition~\ref{p:trans}, Corollary~\ref{c:special-case}, Remark~\ref{rem:deltavanishes} and Proposition~\ref{p:n45} (using Poincaré duality  
as in the proof of Lemma~\ref{l:group-theory2}) that $\Br_{\nr}(V)/\Br_0(V) \cong \Sha^1_{\cyc}(G,\what{\B})$. The desired result now follows from Remark~\ref{r:sha-cyc-is-ba}.
\end{proof}

\begin{example}\label{e:n4}
If $n=4$, $p=2$ and $(\alp_0,\dots,\alp_3)=([ab],[a],[b],[ab])$ for $a,b \in k^*$ such that
none of $a$, $b$, $ab$ is a square, then the classes $[a],[b]$ determine an isomorphism $G = \ZZ/2 \times \ZZ/2$ and
one can calculate that 
\[ \Sha^1_{\cyc}(G,\what{\B}) = \ZZ/2 .\] 
Specifically, in the case of $k=\QQ$ and $(\alp_0,\alp_1,\alp_2,\alp_3) = ([34],[2],[17],[34])$ it can be shown that the non-trivial element of $\Ba_{\omega}(V) = \Sha^1_{\cyc}(G,\what{\B})$ (see Remark~\ref{r:sha-cyc-is-ba}) is in fact locally constant at all places and yields a non-trivial obstruction to the Hasse principle (see~\cite[Example A.15]{GMTW16}). 
We note that this is by no means a contradiction to Theorem~\ref{t:massey}: indeed, an obstruction coming from a locally constant Brauer class means that the homomorphism $\alp\colon \Gam_k \lrar \A = \U/\U^1$ does not lift to $\U/\U^3=\U/\Z$, and hence the relevant Massey product is not defined.
\end{example}

Although we do not include the details here,
it is possible to give a precise description of the unramified Brauer group of~$V$ when $n=6$ as well.
As it turns out, unramified classes that fail to be locally constant at infinitely many places do exist in
this case.
In other words, the statements of Proposition~\ref{p:n45} and of Corollary~\ref{c:n45} both fail when $n=6$.

\section{Beyond the Massey vanishing conjecture}\label{s:beyond}
Higher Massey products can be defined for Galois cohomology classes in more general modules. Indeed, let $n\geq 2$ and suppose that
we are presented with Galois modules $M_{i,j}$ for $0 \leq i < j \leq n$ and with multiplication maps $M_{i,j} \otimes M_{j,k} \lrar M_{i,k}$ for $0 \leq i < j < k \leq n$ satisfying the obvious associativity condition. Given classes $\alp_i \in H^1(\Gam_k,M_{i,i+1})$ for $0\leq i<n$, a \emph{defining system} for the $n$-fold Massey product of $\alp_0,\dots,\alp_{n-1}$ is a collection of $1$-cochains $a_{i,j} \in C^1(\Gam_k,M_{i,j})$ for $(i,j) \neq (0,n)$ such that $\partial a_{i,j} = -\sum_{m=i+1}^{j-1} a_{i,m} \cup a_{m,j}$ (in particular~$a_{i,i+1}$ is a cocycle) and $[a_{i,i+1}] = \alp_i \in H^1(\Gam_k,M_{i,i+1})$. 
Given a defining system $\Lam = \{a_{i,j}\}$ as above, the element $b_{0,n} = -\sum_{m=1}^{n-1}a_{0,m} \cup a_{m,n}$ is a $2$-cocycle, and the value $\<\alp_0,\dots,\alp_{n-1}\>_{\Lam} := [b_{0,n}] \in H^2(\Gam_k,M_{0,n})$ is the \emph{$n$-fold Massey product} of $\alp_0,\dots,\alp_{n-1}$ with respect to the defining system $\Lam$.

\begin{rem}
This set-up was considered by Dwyer \cite[p.~183]{Dwy75} when $\Gam_k$ is replaced with an abstract discrete group and the action on the~$M_{i,j}$ is trivial.
We recall that when working with non-trivial Galois modules the cup product is defined with a twist by the Galois action. More precisely, if $M,M',M''$ are three Galois modules equipped with a map $(-)\cdot(-)\colon M \otimes M' \lrar M''$, then the associated cup product of two $1$-cochains $a\colon\Gam_k \lrar M$ and $a'\colon\Gam_k \lrar M'$ is given by the $2$-cochain $(a \cup a')(\sig,\tau) = a(\sig)\cdot \sigma(a'(\tau))$.
\end{rem}

\begin{example}
\label{example:highermassey}
If $R$ is a commutative ring and $N_0,\dots,N_n$ are Galois $R$-modules,
we can take $M_{i,j} = \Hom_R(N_j,N_i)$ with the induced Galois action.
\end{example}

Classical Massey products with coefficients in~$R$
correspond to the particular case of Example~\ref{example:highermassey} in which $N_0=\dots=N_n=R$ with trivial Galois action.
As another particular case of special interest, let $R=\ZZ/p$ and let $N_i=\Hom({\mmu}_p^{\otimes i},\ZZ/p)$,
so that $M_{i,j}={\mmu}_p^{\otimes \smash{(j-i)}}$.
Thus, for $\alp_0,\dots,\alp_{n-1} \in H^1(k,\mmu_p)$, we obtain $n$-fold Massey products in $H^2(k,\mmu_p^{\otimes n})$.

It is legitimate to wonder about the validity of Conjecture~\ref{c:massey}
in this more general setting.
In this section we will look into this question and try to indicate what can be said from the point of view of the approach applied above to
classical Massey products with coefficients in~$\FF_p$.
For this, we focus our attention on the case of Example~\ref{example:highermassey} where we assume in addition that $R$ and $N_0,\dots,N_n$ are finite---an assumption that will remain in force until the end of this section.

Let $W = \oplus_i N_i$ be considered as an $R$-module (without Galois action). For $m=0,\dots,n$, let $W_m = \oplus_{i=0}^{m} N_i$. Thus $0 = W_{-1} \subseteq W_0 \subseteq W_{1} \subseteq \dots \subseteq W_n=W$ is a filtration of~$W$.
Let $\T(W) \subseteq \Aut_R(W)$ be the subgroup of the group of $R$-linear automorphisms of $W$ which preserves this filtration. Every $Q \in \T(W)$ has an induced $R$-linear action on $N_i = W_i/W_{i-1}$ and so we have a homomorphism $\rho\colon\T(W) \lrar \prod_i \Aut_R(N_i)$. Let $\U(W) \subseteq \T(W)$ denote the kernel of $\rho$. We note that $\U(W)$ is a finite nilpotent group, and that when $R=N_i=\FF_p$ the group $\U(W)$ coincides with the unipotent group $\U$ considered in this paper. Let $\U^1(W) \subseteq \U(W) \subseteq \T(W)$ be the subgroup consisting of those $R$-linear automorphisms $Q\colon W \lrar W$ which act as the identity on $W_i/W_{i-2}$ for every $i=1,\dots,n$. Then $\U^1(W)$ is normal in $\T(W)$ and the quotient group $\A(W) := \T(W)/\U^1(W)$ sits in a canonically split short exact sequence
\[ 1 \lrar \displaystyle\mathop{\bigoplus}_{i=0}^{n-1}M_{i,i+1} \lrar \A(W) \lrar \prod_{i=0}^{n} \Aut_R(N_i) \lrar 1 ,\]
where $M_{i,i+1} := \Hom_R(N_{i+1},N_i)$. Now each $N_i$ is equipped with an $R$-linear Galois action which can be encoded as a homomorphism $\chi\colon\Gam_k \lrar \prod_i \Aut_R(N_i)$. Furthermore, given cohomology classes $\alp_i \in H^1(k,M_{i,i+1})$, a choice
of cocycles representing the~$\alp_i$ determines a lift of $\chi$ to $\alp\colon \Gam_k \lrar \A(W)$. (Choosing other cocycles
corresponds to conjugating~$\alp$ by an element of $\oplus_{i=0}^{n-1}M_{i,i+1}$.) We may then consider the embedding problem

\begin{align}
\begin{aligned}
\label{e:embedding-2}
\xymatrix@R=3ex{
& & & \Gam_k \ar^{\alp}[d] \ar@{-->}[dl] & \\
1 \ar[r] & \U^1(W) \ar[r] & \T(W) \ar[r] & \A(W) \ar[r] & 1 \rlap{.}
}
\end{aligned}
\end{align}
For $i < j$ we may identify homomorphisms $f \in \Hom_R(N_j,N_i)$ with elements $Q_f \in \U(W)$ of the form $Q_f(v_0,\dots,v_n) = (v_0,v_1,\dots,v_i+f(v_j),\dots,v_n)$. 
Consider the subgroup $\Z(W) = M_{0,n} = \Hom_R(N_n,N_0) \subseteq \U^1(W) \subseteq \T(W)$.  
We note that $\Z(W)$ is abelian and normal in $\T(W)$ (and central in~$\U(W)$, though not in~$\T(W)$ in general)
and we may consider the element $u \in H^2(\T(W)/\Z(W),\Z(W))$ which classifies the
group extension
\begin{align}
\label{eq:zttz}
1 \lrar \Z(W) \lrar \T(W) \lrar \T(W)/\Z(W) \lrar 1\rlap{.}
\end{align}
We endow~$\Z(W)$ with the Galois action coming from the equality $\Z(W)=M_{0,n}$ and note that
for any
group homomorphism $\alp_{\Lam}\colon\Gam_k \lrar \T(W)/\Z(W)$ lifting $\alp\colon\Gam_k \lrar \A(W)$,
this action coincides with the one obtained by combining~$\alp_{\Lam}$
with the action of~$\T(W)/\Z(W)$ on~$\Z(W)$ induced by~\eqref{eq:zttz}.

\begin{prop}\label{p:dwyer-2}
Defining systems $\Lam = \{a_{i,j}\}$ for the $n$-fold Massey product of $\{\alp_i \in H^1(k,M_{i,i+1})\}_{i=0}^{n-1}$ are in bijection with group homomorphisms $\alp_{\Lam}\colon\Gam_k \lrar \T(W)/\Z(W)$ lifting $\alp\colon\Gam_k \lrar \A(W)$. Furthermore, the value of the $n$-fold Massey product of $\alp_0,\dots,\alp_{n-1}$ with respect to $\Lam$ is the class $-\alp_{\Lam}^*u \in H^2(k,\Z(W)) = H^2(k,M_{0,n})$.
\end{prop}

\begin{proof}
One argues as in~\cite[Theorem~2.6]{Dwy75}.
For $0 \leq i \leq n$, let $\alpha_{i,i}\colon\Gam_k \lrar \Aut(N_i)$ denote the action of $\Gam_k$.
Suppose that~$\alpha_\Lambda\colon\Gam_k\lrar \T(W)/\Z(W)$
is a group homomorphism lifting the $\alpha_{i,i}$ and denote by $\alpha_{i,j}\colon\Gam_k \lrar M_{i,j}$,
$0 \leq i<j\leq n$, $(i,j)\neq (0,n)$,
the maps it induces, so that
\[\alpha_{i,j}(\sigma\tau)=\sum_{m=i}^j \alpha_{i,m}(\sigma)\alpha_{m,j}(\tau)\rlap{.}\]
Defining $a_{i,j}\colon\Gam_k \lrar M_{i,j}$
by $a_{i,j}(\sigma)=\alpha_{i,j}(\sigma)\alpha_{j,j}(\sigma^{-1})$, we have
\begin{align*}
(\partial a_{i,j})(\sigma,\tau) &= \sigma(a_{i,j}(\tau)) - a_{i,j}(\sigma\tau)+a_{i,j}(\sigma) \\
&=\alpha_{i,i}(\sigma)\alpha_{i,j}(\tau)\alpha_{j,j}(\tau^{-1})\alpha_{j,j}(\sigma^{-1})
-\alpha_{i,j}(\sigma\tau)\alpha_{j,j}((\sigma\tau)^{-1}) + \alpha_{i,j}(\sigma)\alpha_{j,j}(\sigma^{-1})\\
&=-\sum_{m=i+1}^{j-1}\alpha_{i,m}(\sigma)\alpha_{m,j}(\tau)\alpha_{j,j}(\tau^{-1})\alpha_{j,j}(\sigma^{-1})\rlap{,}
\end{align*}
for all $0\leq i<j\leq n$ such that $(i,j)\neq (0,n)$,
and on the other hand
\begin{align}
\begin{aligned}
\label{eq:cupproductaimmj}
(a_{i,m}\cup a_{m,j})(\sigma,\tau) &= a_{i,m}(\sigma) \sigma(a_{m,j}(\tau)) 
= a_{i,m}(\sigma) \alpha_{m,m}(\sigma) a_{m,j}(\tau) \alpha_{j,j}(\sigma^{-1})
\\
&=\alpha_{i,m}(\sigma)\alpha_{m,j}(\tau)\alpha_{j,j}(\tau^{-1})\alpha_{j,j}(\sigma^{-1})\rlap{,}
\end{aligned}
\end{align}
so that
$\partial a_{i,j} = -\sum_{m=i+1}^{j-1} a_{i,m} \cup a_{m,j}$,
i.e., $\{a_{i,j}\}$ forms a defining system.
Reversing this computation
shows that for any defining system $\Lambda=\{a_{i,j}\}$,
if $\alpha_{i,j}\colon\Gam_k\lrar M_{i,j}$
denotes, for $0\leq i<j\leq n$, $(i,j)\neq (0,n)$,
the map defined by $\alpha_{i,j}(\sigma)=a_{i,j}(\sigma)\alpha_{j,j}(\sigma)$,
then the map $\alpha_\Lambda\colon\Gam_k\lrar \T(W)/\Z(W)$ assembled from the~$\alpha_{i,j}$
is a group homomorphism.
This shows the first part of the claim.
To prove the second part, let
$p\colon T(W) \lrar M_{0,n}$
denote the set-theoretic projection
and
$\widetilde\alpha_\Lambda\colon\Gam_k\lrar T(W)$
the set-theoretic lifting of~$\alpha_\Lambda$
such that $p(\widetilde\alpha_\Lambda(\sigma))=0$ for all~$\sigma$.
The class $\alpha_\Lambda^*u$
is represented by
the $2$\nobreakdash-cocycle $b\colon\Gam_k\times \Gam_k \lrar \Z(W)$ defined by
$b(\sigma,\tau)=\widetilde \alpha_\Lambda(\sigma)\widetilde \alpha_\Lambda(\tau) \widetilde\alpha_\Lambda(\sigma\tau)^{-1}$.
Applying~$p$
 to the equality
$b(\sigma,\tau)\widetilde\alpha_\Lambda(\sigma\tau)=\widetilde \alpha_\Lambda(\sigma)\widetilde \alpha_\Lambda(\tau)$ yields
$b(\sigma,\tau)\alpha_{n,n}(\sigma\tau)=\sum_{m=1}^{n-1} \alpha_{0,m}(\sigma)\alpha_{m,n}(\tau)$,
which, by~\eqref{eq:cupproductaimmj}, amounts to
$b(\sigma,\tau) = \sum_{m=1}^{n-1}(a_{0,m} \cup a_{m,n})(\sigma,\tau)$,
i.e., $\<\alp_0,\dots,\alp_{n-1}\>_{\Lam} = [-b] =-\alpha_\Lambda^*u\in H^2(\Gam_k,M_{0,n})$.
\end{proof}

The homogeneous space associated with the (finite) embedding problem~\eqref{e:embedding-2} by the construction already used in~\S\ref{s:dwyerpalschlank}
can again serve as a splitting variety for the $n$-fold Massey product of $\{\alp_i\}$. We may hence, in principle, attempt to apply the strategy of this paper to these generalized Massey products. We begin with the following case, in which the method works with very mild modifications:

\begin{thm}\label{t:massey-2}
Fix $n \geq 3$ and let $R=\FF_p$ for a prime number $p$. For $i=0,\dots,n$, let $N_i$ be a Galois $R$-module whose underlying abelian group is isomorphic to $R$, and set $M_{i,j} = \Hom_R(N_j,N_i)$ for $0 \leq i < j \leq n$. Then the $n$-fold Massey product of any $n$-tuple of classes $\{\alp_i \in H^1(k,M_{i,i+1})\}_{i=0}^{n-1}$ vanishes as soon as it is defined.
\end{thm}
\begin{proof}
The proof is essentially the same as that of Theorem~\ref{t:massey}, and so we will simply indicate the necessary modifications. 
We first note that since $R=\FF_p$ and the underlying abelian group of each $N_i$ is the additive group of $\FF_p$ we may identify $\T(W)$ with the group of upper triangular matrices and $\U(W) \subseteq \T(W)$ with the group of unipotent upper triangular matrices. Furthermore, under this identification the subgroup $\U^1(W) \subseteq \U(W)$ is simply the subgroup of those matrices whose first non-principal diagonal vanishes. In particular, $\U(W)$ and $\U^1(W)$ are the same as the groups $\U,\U^1$ we had before, and so we will simplify notation and write $\T,\U$ and $\U^1$ for $\T(W),\U(W)$ and $\U^1(W)$, respectively. We will also use the notation $\U^m$ for $m=1,\dots,n$ and $\B := \U^1/\U^3$ with the same meaning as in the previous sections.

By the above we see that the homogeneous space $V$ which serves as a splitting variety for~ \eqref{e:embedding-2} has the same geometric stabilizers as the homogeneous space we had for ordinary Massey products, only with a possibly different outer Galois action. We claim that with this outer action the group $\U^1$ is still supersolvable: indeed, the action of $\A$ on each of the quotients $\U^m/\U^{m+1}$ is diagonalizable, with eigenspaces given by the cyclic subgroups $\Hom_R(N_{i+m},N_{i})$. In particular, we may still use~\cite[Théorème~B]{HW18} to deduce the existence of rational points on $V$ when given a collection of local points which is orthogonal to the unramified Brauer group. In addition, Proposition~\ref{p:trans} equally applies in this case, showing that $V$ has a trivial transcendental unramified Brauer group. We now claim that Proposition~\ref{p:bound_on_brnr1} holds as well. Let $\B_0 \subseteq \B$ be as in Proposition~\ref{p:bound_on_brnr1}. By the inflation-restriction exact sequence we see that the statement of Proposition~\ref{p:bound_on_brnr1} is equivalent to the statement that the composed map $\Br_{1,\nr}(V) \lrar H^1(k,\what{\B}) \lrar H^1(k,\what{\B}_0)$ is the zero map. Let $L \subset \ovl{k}$ be the subfield
fixed by the kernel of $\chi\colon \Gam_k \lrar \prod_i \Aut(N_i)$. By our assumption $\Aut(N_i) \cong \FF_p^*$ has order prime to $p$ and so $[L:k]$ is prime to $p$.  Since $\what{\B}_0$ is a $p$-torsion module it follows that the map $H^1(k,\what{\B}_0) \lrar H^1(L,\what{\B}_0)$ is injective, and hence to prove the statement we may as well extend our scalars to $L$. But now we are reduced to the case of ordinary Massey products to which Proposition~\ref{p:bound_on_brnr1} itself applies.

Arguing as in the proof of Theorem~\ref{t:massey}, it will now suffice to
prove an analogue of Proposition~\ref{p:local}.
Let $r,s \in \{1,\dots,n-2\}$
satisfy $r+s=n-1$ and set
\[ \P^{r,s} = \bigoplus_{i=0}^{r} \Hom_R(N_n,N_i) \oplus \bigoplus_{j=n-s}^{n} \Hom_R(N_j,N_0) \subseteq \U^1\rlap{.}\]
If~$v$ is a place of~$k$, we need to prove
that a homomorphism $\Gam_{k_v} \lrar \T/\P^{r,s}$ that lifts to $\Gam_{k_v} \lrar \T/\Z$ also
lifts to $\Gam_{k_v} \lrar \T$.
As $\P^{r,s}$ is abelian,
the obstruction to the existence of such a lifting lives in $H^2(k_v,\P^{r,s})$.
As $\P^{r,s}$ has exponent~$p$ and as $[L:k]$ is prime to~$p$,
the restriction map $H^2(k_v,\P^{r,s}) \to H^2(L \otimes_k k_v,\P^{r,s})$ is injective;
we may therefore once again extend the scalars from~$k$ to~$L$ and apply Proposition~\ref{p:local} directly.
\end{proof}

In contrast with Theorem~\ref{t:massey-2}, we will now show that when $R=\ZZ/8$, triple Massey products may be defined and non-trivial, even when $N_i = R$ for all $i$.

\begin{thm}\label{t:counter}
There exist
$\chi_{0,1},\chi_{1,2},\chi_{2,3} \in H^1(\QQ,\ZZ/8)$
such that the triple Massey product
$\<\chi_{0,1},\chi_{1,2},\chi_{2,3}\>$ is defined but does not contain $0$.
\end{thm}

The remainder of this section is devoted to the proof of Theorem~\ref{t:counter}. Let us choose distinct positive primes $p,q$ which are both $1$ mod $8$.

\begin{lem}
The fields $\QQ(\sqrt{p})$ and $\QQ(\sqrt{2q})$
can be embedded in cyclic extensions of~$\QQ$ of degree~$8$.
In other words, the classes of these quadratic extensions
belong to the image of the natural map $H^1(\QQ,\ZZ/8)\to H^1(\QQ,\ZZ/2)$.
\end{lem}

\begin{proof}
Let us check that for any place~$v$ of~$\QQ$,
the images of these classes in $H^1(\QQ_v,\ZZ/2)$ come from $H^1(\QQ_v,\ZZ/8)$.
For $v \notin \{2,p,q\}$ this is clear, as these quadratic extensions are unramified at~$v$
if~$v$ is finite and split at~$v$ if~$v$ is real.
For $v \in \{2,p,q\}$,
we need to see,
by local duality, that the images of these two classes
are orthogonal to the kernel of the natural map $H^1(\QQ_v,\mmu_2) \to H^1(\QQ_v,\mmu_8)$;
that is, that the Hilbert symbols $(p,x)_v$ and $(2q,x)_v$ are trivial
for all $x \in \mmu_4(\QQ_v)$.
If $v \in \{p,q\}$, this is true as~$x$ is automatically a square in~$\QQ_v$.
If $v=2$, then $x \in \mmu_2(\QQ_v)$, so that $(2,x)_v$ is trivial, and $p$ and~$q$ are squares in~$\QQ_v$.

Any rational number which is everywhere locally a $4$\nobreakdash-th power is a $4$\nobreakdash-th power.
Hence $\Sha^1(\QQ,\mmu_4)=0$, from which it follows, by Poitou--Tate duality,
that $\Sha^2(\QQ,\ZZ/4)=0$.  In view of the exact sequence $0 \to \ZZ/4 \to \ZZ/8 \to \ZZ/2 \to 0$,
the lemma follows.
\end{proof}

Let us fix a homomorphism $\chi_{0,1}\colon \Gam_{\QQ} \lrar \ZZ/8$ lifting the quadratic character of $\QQ(\sqrt{p})/\QQ$
and a homomorphism $\chi_{2,3}\colon \Gam_{\QQ} \lrar \ZZ/8$ lifting the quadratic character of $\QQ(\sqrt{2q})/\QQ$.

Consider the group $\U_{\ZZ/8}$ of upper triangular unipotent $4\times 4$-matrices with coefficients in $\ZZ/8$. Let $\P \subseteq \U_{\ZZ/8}$ be the subgroup generated by $e_{0,2},e_{0,3},e_{1,2},e_{1,3}$, so that $\P$ is abelian and contains the subgroup $\U^1_{\ZZ/8} \subseteq \U_{\ZZ/8}$ of upper triangular unipotent matrices whose first non-principal diagonal vanishes. Let $\C := \U_{\ZZ/8}/\P$ so that $\C$ is abelian generated by the images $\ovl{e}_{0,1},\ovl{e}_{2,3}$ of $e_{0,1}$ and $e_{2,3}$ respectively. Using this basis, the cyclic characters $\chi_{0,1},\chi_{2,3}\colon \Gam_\QQ \lrar \ZZ/8$ can be interpreted as a single homomorphism $(\chi_{0,1},\chi_{2,3})\colon\Gam_{\QQ} \lrar \C$. Since $\P$ is abelian we have an honest action of $\C$ on $\P$ which we can pull back to an action of $\Gam_{\QQ}$ on $\P$. We note that this action preserves $\U^1_{\ZZ/8}$ and we have a short exact sequence of Galois modules
\[ 1 \lrar  \U^1_{\ZZ/8} \lrar \P \lrar \ZZ/8\<\ovl{e}_{1,2}\> \lrar 1 .\]
We then claim the following:
\begin{lem}\label{l:sha-2}
The natural map $\Sha^2(\QQ,\U^1_{\ZZ/8}) \to \Sha^2(\QQ,\P)$ is not injective.
\end{lem}
\begin{proof}
By global arithmetic duality it will suffice to show that the map
\[ \Sha^1(\QQ,\what{\P}) \lrar \Sha^1(\QQ,\what{\U}^1_{\ZZ/8}) \] 
is not surjective. Consider the short exact sequence of Galois modules
\[ 0 \lrar \mmu_8 \times \mmu_8 \lrar \what{\U}^1_{\ZZ/8} \lrar \mmu_8 \lrar 0 \]
where the inclusion $\mmu_8 \times \mmu_8 \subseteq \what{\U}^1_{\ZZ/8}$ is dual to the projection $\U^1_{\ZZ/8} \lrar \<e_{0,2},e_{1,3}\>$ and the projection $\what{\U}^1_{\ZZ/8} \lrar \mmu_8$ is dual to the inclusion $\<e_{0,3}\> \subseteq \U^1_{\ZZ/8}$. For $a,b \in \QQ^*$ we may consider the corresponding class $([a],[b]) \in H^1(\QQ,\mmu_8 \times \mmu_8) = \QQ^*/(\QQ^*)^8 \times \QQ^*/(\QQ^*)^8$. Unwinding the definitions we see that for any field $K$ containing $\QQ$, the boundary map
\[ H^0(K,\mmu_8) \lrar H^1(K,\mmu_8 \times \mmu_8) \]
sends $-1 \in \mmu_8(K)$ to $([16q^4],[p^4])$. 
This implies that $([16],[1]) \in H^1(\QQ,\mmu_8 \times \mmu_8)$ maps to an element of $\Sha^1(\QQ,\what{\U}^1_{\ZZ/8}) \subseteq H^1(\QQ,\what{\U}^1_{\ZZ/8})$. Indeed, for every place $v \neq 2$ the element $16$ is an $8$\nobreakdash-th power in $\QQ_v$ and hence the element $([16]_v,[1]_v)$ is trivial, and at $v=2$ we have that $([16]_v,[1]_v) = ([16q^4],[p^4])$. We note, however, that the image of $([16],[1])$ in $\Sha^1(\QQ,\what{\U}^1_{\ZZ/8})$ is non-zero since $\mmu_8(\QQ) = \{1,-1\}$ and $([16],[1])$ is not a multiple of $([16q^4],[p^4])$ in $H^1(\QQ,\mmu_8 \times \mmu_8)$. We now claim the following:
\begin{center}
$(*)\quad$ The image of $([16],[1])$ in $\Sha^1(\QQ,\what{\U}^1_{\ZZ/8})$ does not come from $\Sha^1(\QQ,\what{\P})$. 
\end{center}
The homomorphism $(\chi_{0,1},\chi_{2,3})\colon\Gam_\QQ \lrar \C$ is surjective. Indeed, the fact that $[p],[2q]$ are linearly independent in $H^1(\QQ,\ZZ/2)$ implies that $\chi_{0,1},\chi_{2,3}$ are linearly independent in $H^1(\QQ,\ZZ/8)$, so that the image of this homomorphism
is not contained in any index~$2$ subgroup of~$\C$.
Let $L \subset \ovl{k}$ be the subfield fixed by the kernel of this homomorphism.  We henceforth identity $\Gal(L/\QQ)$ with $\C$. Consider the commutative diagram
\[ \xymatrix@R=3ex{
\Sha^1(\QQ,\what{\P}) \ar[r]\ar[d] & \Sha^1(\QQ,\what{\U}^1_{\ZZ/8}) \ar[d] \\
\Sha^1(L,\what{\P})^{\C} \ar[r] & \Sha^1(L,\what{\U}^1_{\ZZ/8})^{\C} \rlap{.}
}\]
To prove $(*)$ it will suffice to show that the image of $([16],[1])$ in $\Sha^1(L,\what{\U}^1_{\ZZ/8})^\C$ does not come from $\Sha^1(L,\what{\P})^\C$. Unwinding the definitions we may identify $\Sha^1(L,\what{\U}^1_{\ZZ/8})^\C$ with the group of triples 
\[ ([a_{0,2}],[a_{0,3}],[a_{1,3}]) \in (\Sha^1(L,\mmu_8))^3 \subseteq L^*/(L^*)^8 \times L^*/(L^*)^8 \times L^*/(L^*)^8 \] 
such that $\ovl{e}_{0,1}[a_{1,3}] = [a_{0,3}][a_{1,3}]$,
$\ovl{e}_{0,1}[a_{0,j}] = [a_{0,j}]$
for $j \in \{2,3\}$,
 $\ovl{e}_{2,3}[a_{0,2}] = [a_{0,2}][a_{0,3}^{-1}]$,
and $\ovl{e}_{2,3}[a_{i,3}] = [a_{i,3}]$
for $i \in \{0,1\}$,
 where $\ovl{e}_{i,j} \in \C$ acts on $L^*/(L^*)^8$ via the identification $\C = \Gal(L/\QQ)$. Similarly, we may identify $\Sha^1(L,\what{\P})^\C$ with the group of quadruples 
\[ ([a_{0,2}],[a_{0,3}],[a_{1,2}],[a_{1,3}]) \in (\Sha^1(L,\mmu_8))^4 \subseteq
L^*/(L^*)^8 \times L^*/(L^*)^8 \times L^*/(L^*)^8 \times L^*/(L^*)^8 \]
such that
$\ovl{e}_{0,1}[a_{1,j}] = [a_{0,j}][a_{1,j}]$ for $j\in\{2,3\}$,
$\ovl{e}_{0,1}[a_{0,j}] = [a_{0,j}]$ for $j \in \{2,3\}$,
$\ovl{e}_{2,3}[a_{i,2}] = [a_{i,2}][a_{i,3}^{-1}]$ for $i\in\{0,1\}$,
and $\ovl{e}_{2,3}[a_{i,3}] = [a_{i,3}]$ for $i \in \{0,1\}$.
We wish to show that the triple $([16],[1],[1])$ is not in the image of the map
\begin{equation}\label{e:map}
\Sha^1(L,\what{\P})^\C \lrar \Sha^1(L,\what{\U}^1_{\ZZ/8})^\C\text{,} \quad ([a_{0,2}],[a_{0,3}],[a_{1,2}],[a_{1,3}]) \mapsto ([a_{0,2}],[a_{0,3}],[a_{1,3}])\rlap{.}
\end{equation} 
We note that $[16]$ is indeed a non-zero element of $\Sha^1(L,\mmu_8)$: $16$ is not an $8$\nobreakdash-th power in~$L$ since $L$ does not contain either of $\sqrt{2},\sqrt{-2},\sqrt{-1}$; indeed, the quadratic subextensions of $L$ are by construction $\QQ(\sqrt{p}),\QQ(\sqrt{2q})$ and $\QQ(\sqrt{2qp})$, and $p\neq q$. To finish the proof it will suffice to show that there is no element $[a_{1,2}] \in \Sha^1(L,\mmu_8)$ such that $\ovl{e}_{0,1}[a_{1,2}] = [16][a_{1,2}]$. But this is now a consequence of the Grunwald-Wang theorem, which says, in particular, that for every number field $K$ the group $\Sha^1(K,\mmu_m)$ is either trivial or $\ZZ/2$ (see, e.g., \cite[Chapter Ten, Theorem 1]{AT52}). In our case we have a non-zero element $[16]$ and so $\Sha^1(L,\mmu_8) \cong \ZZ/2$, generated by $[16]$. In particular, the action of $\C$ on $\Sha^1(L,\mmu_8)$ is trivial and there can be no element $[a_{1,2}] \in \Sha^1(L,\mmu_8)$ such that $\ovl{e}_{0,1}[a_{1,2}] = [16][a_{1,2}]$.
\end{proof}

\begin{proof}[Proof of Theorem~\ref{t:counter}]
By Lemma~\ref{l:sha-2}, we may fix a non-zero element $\gam \in \Sha^2(\QQ,\U^1_{\ZZ/8})$ whose image in $H^2(\QQ,\P)$ is trivial. It follows that there exists a cyclic character $\chi_{1,2}\colon \Gam_{\QQ} \lrar \ZZ/8\<\ovl{e}_{1,2}\>$ such that $\partial\chi_{1,2} = \gam \in H^2(\QQ,\U^1_{\ZZ/8})$. 
Let us now prove that the triple Massey product
$\<\chi_{0,1},\chi_{1,2},\chi_{2,3}\>$ is defined but does not contain $0$.

The classes $\chi_{0,1},\chi_{1,2},\chi_{2,3}$ determine a homomorphism $\chi\colon \Gam_\QQ \lrar \ZZ/8\<\ovl{e}_{0,1},\ovl{e}_{1,2},\ovl{e}_{2,3}\>$. Consider the associated Massey embedding problem
\begin{align}
\begin{aligned}
\label{e:embedding-3}
\xymatrix@R=3ex{
& & & \Gam_\QQ \ar^{\chi}[d] \ar@{-->}[dl] & \\
1 \ar[r] & \U^1_{\ZZ/8} \ar[r] & \U_{\ZZ/8} \ar[r] & \ZZ/8\<\ovl{e}_{0,1},\ovl{e}_{1,2},\ovl{e}_{2,3}\> \ar[r] & 1 \rlap{.}
}
\end{aligned}
\end{align}
We claim that~\eqref{e:embedding-3} has a local solution when restricted to every $\Gam_v \subseteq \Gam_\QQ$ but does not have a solution globally.
To see this, consider the commutative diagram of finite groups
\begin{align}\begin{aligned}
\label{e:U-8}
\xymatrix@R=4ex{
\U^1_{\ZZ/8} \ar@{^{(}->}[r]\ar@{=}[d] & \P \ar@{^{(}->}[d]\ar@{->>}[r] & \ZZ/8\<\ovl{e}_{1,2}\> \ar@{^{(}->}[d] \\
\U^1_{\ZZ/8} \ar@{^{(}->}[r] & \U_{\ZZ/8} \ar@{->>}[r]\ar@{->>}[d] & \ZZ/8\<\ovl{e}_{0,1},\ovl{e}_{1,2},\ovl{e}_{2,3}\> \ar@{->>}[d]  \\
& \ZZ/8\<\ovl{e}_{0,1},\ovl{e}_{2,3}\> \ar@{=}[r]\ar@/_0.8pc/[u] & \ZZ/8\<\ovl{e}_{0,1},\ovl{e}_{2,3}\> \ar@/_0.8pc/[u] \\
}
\end{aligned}
\end{align}
in which the rows are exact and the columns are split exact. Here the chosen sections of the bottom vertical projections send the generators $\ovl{e}_{0,1},\ovl{e}_{2,3} \in \ZZ/8\<\ovl{e}_{0,1},\ovl{e}_{2,3}\>$ to $e_{0,1},e_{2,3} \in \U_{\ZZ/8}$ and to $\ovl{e}_{0,1},\ovl{e}_{2,3} \in \ZZ/8\<\ovl{e}_{0,1},\ovl{e}_{1,2},\ovl{e}_{2,3}\>$. Using these sections as base points, we see that homomorphisms $\Gam_\QQ \lrar \U_{\ZZ/8}$ which lift $(\chi_{0,1},\chi_{2,3})\colon \Gam_\QQ \lrar \ZZ/8\<\ovl{e}_{0,1},\ovl{e}_{2,3}\>$ are classified (up to conjugation by~$\P$) by elements in $H^1(\QQ,\P)$, while homomorphisms $\Gam_\QQ \lrar \ZZ/8\<\ovl{e}_{0,1},\ovl{e}_{1,2},\ovl{e}_{2,3}\>$ which lift $(\chi_{0,1},\chi_{2,3})\colon \Gam_\QQ \lrar \ZZ/8\<\ovl{e}_{0,1},\ovl{e}_{2,3}\>$ are in bijection with the elements of $H^1(\QQ,\ZZ/8\<\ovl{e}_{1,2}\>)$. The map between these two sets of lifts induced by the projection $\U_{\ZZ/8} \lrar \ZZ/8\<\ovl{e}_{0,1},\ovl{e}_{1,2},\ovl{e}_{2,3}\>$ is compatible with the map on Galois cohomology induced by $\P \lrar \ZZ/8\<\ovl{e}_{1,2}\>$. We conclude that $\chi\colon\Gam_\QQ \lrar \ZZ/8\<\ovl{e}_{0,1},\ovl{e}_{1,2},\ovl{e}_{2,3}\>$ lifts to $\Gam_\QQ \lrar \U_{\ZZ/8}$ if and only if $\chi_{1,2} \in H^1(\QQ,\ZZ/8\<\ovl{e}_{1,2}\>)$ comes from $H^1(\QQ,\P)$. By our choice of $\chi_{1,2}$ the last statement holds when base changing to each completion of $\QQ$, but not for $\QQ$ itself. It then follows that the embedding problem~\eqref{e:embedding-3} is solvable everywhere locally, but not globally.

To finish the proof, we note that the local solvability of~\eqref{e:embedding-3} implies in particular that the elements $\chi_{0,1} \cup \chi_{1,2},\chi_{1,2} \cup \chi_{2,3} \in H^2(\QQ,\ZZ/8)$ both belong to $\Sha^2(\QQ,\ZZ/8)$. As $\Sha^1(\QQ,\mmu_8)=0$, we have $\Sha^2(\QQ,\ZZ/8) = 0$ by global duality. We conclude that the triple Massey product $\<\chi_{0,1},\chi_{1,2},\chi_{2,3}\>$ is indeed defined. Nonetheless, it cannot contain~$0$ since the embedding problem~\eqref{e:embedding-3} is not solvable.
\end{proof}

\bibliographystyle{amsalpha}
\bibliography{massey}
\end{document}